\documentclass[11pt]{article}
\usepackage{amsmath, graphicx, amsfonts,amssymb, calrsfs}
\usepackage{amsfonts}
\usepackage{mathrsfs}
\usepackage{amsthm}

\usepackage[pdfstartview=FitH,
            CJKbookmarks=true,
            bookmarksnumbered=true,
            bookmarksopen=true,
            colorlinks=true,
            linkcolor=blue,
            anchorcolor=blue,
            citecolor=red,
            urlcolor=blue
            ]{hyperref}

\numberwithin{equation}{section}
\newtheorem{theorem}{Theorem}[section]
\newtheorem{definition}[theorem]{Definition}
\newtheorem{lemma}[theorem]{Lemma}
\newtheorem{corollary}[theorem]{Corollary}

\newtheorem{proposition}[theorem]{Proposition}

\newtheorem{problem}[theorem]{Problem}

\newcommand{\ee}{\varepsilon}
\newcommand{\R}{\mathbb{R}}
\newcommand{\Rn}{\mathbb{R}^n}
\def\cK{\mathcal{K}}
\def\cKn{\mathcal{K}^n}
\def\cKo{\mathcal{K}_o^n}
\def\cKon{\mathcal{K}_{(o)}^n}
\def\sphere{S^{n-1}}
\def\ball{B^n}
\def\dveV{\widetilde{V}_G}

\def\N{\mathbb{N}}
\def\into{\mathrm{int}}
\def\cH{\mathcal{H}^{n-1}}
\def\dveV{\widetilde{V}_G}

\def\cG{\mathcal{G}}
\def\hG{\widehat{V}_{G}}
\def\cBt{\widetilde{\mathcal{B}}}
\def\cBh{\widehat{\mathcal{B}}}
\def\bC{\mathbf{C}}
\def\cI{\mathcal{I}}
\def\cD{\mathcal{D}}
\def\vG{V_G}
\def\vGb{\overline{V}_G}
\def\cB{\mathcal{B}}
\def\cBb{\overline{\mathcal{B}}}
\def\cA{\mathcal{A}}

\def\bp{\begin{proposition}}
\def\ep{\end{proposition}}
\def\be{\begin{equation}}
\def\ee{\end{equation}}
\def\bd{\begin{definition}}
\def\ed{\end{definition}}

\title{The general dual-polar Orlicz-Minkowski problem  }

 \author{Sudan Xing, Deping Ye and Baocheng Zhu}

\begin{document}
\date{}
\maketitle

\begin{abstract} This paper  gives a systematic study to the general dual-polar Orlicz-Minkowski problem (e.g., Problem \ref{general-dual-polar}). This problem  involves the general dual  volume $\dveV(\cdot)$ recently proposed in \cite{GHWXY, GHXY} in order to study the  general dual Orlicz-Minkowski problem. As $\dveV(\cdot)$ extends the volume and the $q$th dual volume,  the general dual-polar Orlicz-Minkowski problem is ``polar" to the  recently initiated general dual Orlicz-Minkowski problem in \cite{GHWXY, GHXY} and ``dual" to the newly proposed polar Orlicz-Minkowski problem in \cite{LuoYeZhu}. The existence, continuity and uniqueness, if applicable, for the solutions to the general dual-polar Orlicz-Minkowski problem are established. Polytopal solutions and/or counterexamples to the general dual-polar Orlicz-Minkowski problem for discrete measures are also provided. Several variations of the general dual-polar Orlicz-Minkowski problem are discussed as well, in particular  the one leading to the general Orlicz-Petty bodies.

  \end{abstract}

\section{Introduction}

Lutwak's discovery of the $L_p$ surface area measure and the $L_p$ mixed volume \cite{L3} for $p>1$  gave a new and thriving life to the Brunn-Minkowski theory.  Among those fundamental objects related to the $L_p$ surface area measure and the $L_p$ mixed volume, the $L_p$ Minkowski problem (for $p=1$ in \cite{Minkowski1987, Minkowski1903} by Minkowski and for $p\neq 1$ in  \cite{L3} by Lutwak) and the $L_p$ affine surface area  (for $p=1$ in \cite{Blaschke1923} by Blaschke, for $p>1$ in \cite{Lutwak1996} by Lutwak and for $p<1$ in \cite{CE2004} by Sch\"{u}tt and Werner) arguably  have the  greatest influence. The former one aims to find convex bodies (i.e., convex compact sets in $\Rn$ with nonempty interiors) so that their $L_p$ surface area measures coincide with a pre-given nonzero finite Borel measure $\mu$ defined on the unit sphere $\sphere$.  The $L_p$ Minkowski problem has attracted tremendous attention in different areas, such as analysis, convex geometry, and partial differential equations  (see e.g., \cite{BLYZ2013, Chen2006, CW2006, HMS2004, HugLYZ, EDG2004, GZhu2015I, GZhupreprint} among others).  In particular, it is closely related to the far-reaching optimal mass transportation problem via the Monge-Amp\`{e}re type equations. Solutions to the $L_p$ Minkowski problem have been used to develop the powerful tool of convexification for Sobolev functions and to establish the elegant $L_p$ affine Sobolev inequalities as well as the related P\'{o}lya-Szeg\H{o} principles, see e.g., \cite{CLYZ-1, HS-JFA-1, HSXasymmetric, LYZ02, zh99b, GZhang2007}. The latter one (i.e., the $L_p$ affine surface area) is more on the differential properties of convex bodies. It has many beautiful properties, including the affine invariant valuation and being $0$ for polytopes (if $p>0$); these properties make the $L_p$ affine surface areas   perfect geometric invariants in characterizing the affine valuations, the $L_p$ affine isoperimetric inequalities, and approximation of convex bodies by polytopes \cite{Gruber1993, LR1999, LudR, LSW2006, CE1807, WY2008}. The elegant integral expression for the $L_p$ affine surface area also leads to nice observations of its connection with the $f$-divergence \cite{JW2014, GE2012, Werner2012}. It is worth to mention that the celebrated Blaschke-Santal\'{o} inequality was originally established as a consequence of the combination of the solutions to the $L_p$ Minkowski problem and  the affine isoperimetric inequalities for the $L_p$ affine surface area  (in particular, with both $p=1$) (see e.g., \cite{Sch} for details). In words, the importance of the $L_p$ Minkowski problem and the $L_p$ affine surface area can never be over-emphasized.

The $L_p$ Minkowski problem and the $L_p$ affine surface area were apparently developed in completely different approaches, however, they were nicely connected through the $L_p$ geominimal surface area and  the $L_p$ Petty bodies \cite{Lutwak1996, Ye2015, ZHY2016}.  As the bridge to connect several geometries (affine, Minkowski and relative), the $L_p$ geominimal surface area is crucial in convex geometry and, in particular, share many properties similar to those for the $L_p$ affine surface area.  Let $\cKon$ be the set of convex compact sets in $\Rn$ with the origin $o$ in their interiors.  Finding the $L_p$ Petty bodies of $K\in \cKon$ for $p\in \R\setminus\{0, -n\}$ requires to solve the following optimization problem (with $\mu$ being the $L_p$ surface area measure of $K$):  \begin{equation}\label{p-polar-04-07-19} \inf /\sup\bigg\{\int_{\sphere}h^p_{L^\circ}(u) \,d\mu(u): \ \ L\in \cKon \ \ \mathrm{and}\ \  V(L)=V(\ball)\bigg\},\end{equation} where $\ball$ is the unit Euclidean ball in $\Rn$, $V(\cdot)$ stands for the volume, $L^\circ$ denotes the polar body of $L\in \cKon$, and  $h_L$ is the support function of $L$ (see Section \ref{Section-2} for notations).  As explained in \cite{LuoYeZhu},  the $L_p$ Minkowski problem can be viewed as the ``polarity" of  (\ref{p-polar-04-07-19}) (in particular, for $\mu$ nice enough such as $\mu$ being even)  aiming to find convex bodies (ideally in $\cKon$) to solve the optimization problem similar to (\ref{p-polar-04-07-19}), namely with $L^\circ$ replaced by $L$.  On the other hand, the $L_p$ affine surface area of $K\in \cKon$ can be defined through a formula similar to (\ref{p-polar-04-07-19}) for $\mu$ being the $L_p$ surface area measure of $K$, but with $L\in \cKon$ and $h_{L^\circ}$ replaced by $L$ belong to star bodies about the origin and, respectively,  $\rho_L^{-1}$ where $\rho_L$ is the radial function of $L$  (see \cite{Lutwak1996, Ye2015, ZHY2016} for more details).

The main purpose of this article is to give a systematic study to the general dual-polar Orlicz-Minkowski problem, which extends problem (\ref{p-polar-04-07-19}) in the arguably most general way: with the function $t^p$ (from the integrand of the objective functional)  and $V(L)$ in problem (\ref{p-polar-04-07-19}) replaced by a (general nonhomogeneous) continuous function $\varphi: (0, \infty)\rightarrow (0, \infty)$ and, respectively, $\dveV(L)$, the general dual volume  of $L$,  formulated by  $$\dveV(L)=\int_{S^{n-1}}G(\rho_L(u),u)\,du$$ with  $\,du$ the spherical measure of $\sphere$. Namely, we pose the following problem: {\em Under what conditions on a nonzero finite Borel measure $\mu$ defined
on $\sphere$, continuous functions $\varphi: (0, \infty)\rightarrow (0, \infty)$ and
$G: (0, \infty)\times \sphere\rightarrow (0, \infty)$  can we find a convex body $K\in\cKon$ solving the following optimization problem:}
 \begin{eqnarray}\label{dual-polar-4-12-19}
 \inf /\sup \left\{ \int_{\sphere}\varphi(h_{Q^\circ}(u))d\mu(u): Q\in \cKon \ \ \mathrm{and} \ \ \dveV(Q)=\dveV(\ball)\right\}.
 \end{eqnarray} In particular, problem (\ref{dual-polar-4-12-19}) becomes  problem (\ref{p-polar-04-07-19})  when $\varphi(t)=t^p$ and $G(t, u)=t^n/n$. Moreover, problem (\ref{dual-polar-4-12-19}) also contains as a special case the recent polar Orlicz-Minkowski problem introduced in \cite{LuoYeZhu} by Luo, Ye and Zhu, i.e., solving the following optimization problem: \begin{eqnarray}\label{1-dual-polar-4-12-19}
 \inf /\sup \left\{ \int_{\sphere}\varphi(h_{Q^\circ}(u))d\mu(u): Q\in \cKon \ \ \mathrm{and} \ \ V(Q)=V(\ball)\right\}.
 \end{eqnarray} Note that closely related to (\ref{1-dual-polar-4-12-19}) are the Orlicz affine and geominimal surface areas, which were proposed in \cite{Ye2015c, SHG2015, ZHY2016}.   In fact, one can observe that  (\ref{dual-polar-4-12-19}) not only generalizes  (\ref{1-dual-polar-4-12-19}), but also is ``dual" to  (\ref{1-dual-polar-4-12-19}).  This is one of our motivations to  study the  general dual-polar Orlicz-Minkowski problem.  

Another motivation for our general dual-polar Orlicz-Minkowski problem is its relation and close connection with the recent general dual Orlicz-Minkowski problem in \cite{GHWXY} by Gardner, Hug, Weil, Xing and Ye, and in \cite{GHXY} by Gardner, Hug, Xing and Ye. Indeed, the fundamental geometric invariant $\dveV(\cdot)$ was mainly introduced to derive the general dual Orlicz curvature measures, the key ingredients of the general dual Orlicz-Minkowski problem.  Such Minkowski type problem extends not only the $L_p$ Minkowski problem by Lutwak \cite{L3} and its Orlicz counterpart by Haberl, Lutwak, Yang and Zhang  \cite{HLYZ-2010}, but also the recently initiated dual Minkowski problem by Huang, Lutwak, Yang and Zhang \cite{ LYZActa}, the $L_p$ dual Minkowski problem by  Lutwak, Yang and Zhang \cite{LYZ-Lp}, the dual Orlicz-Minkowski problem by Zhu, Xing and Ye \cite{ZSY2017}, and the general dual Orlicz-Minkowski problem by Xing and Ye \cite{XY2017-1}. Here we would like to emphasize the elegance and significance of the groundbreaking work \cite{ LYZActa}, where the authors, at the first time, proved the far-reaching variational formula for the $q$th dual volume (i.e., the case when $G(t, u)=t^q/n$ for $q\neq 0$) in terms of the logarithmic addition. Such variational formula  can be viewed as a perfect vinculum to deeply connect the two closely related but quite different branches of convex geometry: the $L_p$ Brunn-Minkowski theory for convex bodies and its dual theory for star bodies. The variational formula has been quickly extended to other cases such as \cite{LYZ-Lp, XY2017-1, ZSY2017}, and achieves its most generality when the $q$th volume and the logarithmic addition are replaced by the general dual volume $\dveV(\cdot)$ and an Orlicz addition involving $\varphi$, respectively, in \cite{GHWXY}.  In many circumstance, solving the  general dual Orlicz-Minkowski problem requires to find solutions to the following optimization problem:  \begin{eqnarray}\label{dual--4-12-19}
 \inf /\sup \left\{ \int_{\sphere}\varphi(h_{Q}(u))d\mu(u): Q\in \cKon \ \ \mathrm{and} \ \ \dveV(Q)=\dveV(\ball)\right\}.
 \end{eqnarray} In particular, if $G(t, u)=t^n/n$, (\ref{dual--4-12-19}) recovers the Orlicz-Minkowski problem \cite{HLYZ-2010}. 
In view of (\ref{dual-polar-4-12-19}), one sees that  (\ref{dual-polar-4-12-19}) is ``polar" to (\ref{dual--4-12-19}). It is our belief that, like the general dual Orlicz-Minkowski problem,  the newly proposed general dual-polar Orlicz-Minkowski problem will constitute one of the core objectives in the rapidly developing dual Orlicz-Brunn-Minkowski theory recently started from the work \cite{ghwy15} by Gardner, Hug, Weil and Ye, and independently the work \cite{Zhub2014} by Zhu, Zhou and Xu. 

Our paper is organized as follows. Section \ref{Section-2} provides  a brief collection of notations and well-known facts from convex geometry. In Section \ref{Section:3}, we will introduce the homogeneous general dual volume,  $\hG(\cdot)$, a geometric invariant sharing properties rather similar to those for the general dual volume $\dveV(\cdot)$. Properties of $\hG(\cdot)$, such as, the homogeneity, continuity and monotonicity, are proved in Proposition \ref{properties-G-hat}. Lemma \ref{lemmaforinterior} provides reasonable conditions on $G: (0, \infty)\times \sphere\rightarrow (0, \infty)$ such that, roughly speaking,  if $Q_i\rightarrow Q_0$ in the Hausdorff metric with $Q_i\in \cKon$ for each $i\geq 1$ and  $\{\dveV(Q_i^\circ)\}_{i\geq 1}$ (or $\{\hG(Q_i^\circ)\}_{i\geq 1}$, respectively) as a sequence of real numbers  is bounded, then $Q_0\in \cKon$. This lemma is the key tool to show the existence of solutions to our general dual-polar Orlicz-Minkowski problem (i.e., (\ref{dual-polar-4-12-19})).

 Section \ref{section-4} dedicates to establish the continuity, uniqueness, and existence of solutions to the general dual-polar Orlicz-Minkowski problem. In particular, we first obtain the polytopal solutions to the general dual-polar Orlicz-Minkowski problem when the measure $\mu$ is discrete under certain conditions such as  $\varphi$ being increasing and the infimum in (\ref{dual-polar-4-12-19}) being considered; the detailed statements can be found in Theorem \ref{polardualorlicz1}. In Proposition \ref{denial of other possibilities}, the nonexistence of solutions to the general dual-polar Orlicz-Minkowski problem for discrete measures are proved by counterexamples if the supremum in (\ref{dual-polar-4-12-19}) is
considered,  or if  the infimum is considered with $\varphi$ being decreasing. As $\dveV(\cdot)$ and $\hG(\cdot)$ are not invariant under volume-preserving linear transforms on $\Rn$, our calculations in  Proposition \ref{denial of other possibilities} are more delicate than those in  \cite{LuoYeZhu} where  the volume is considered. Our main results are given in Theorem \ref{polardualorlicz-1-29-1} and Corollary \ref{con-4-12-19}, where the existence,  uniqueness and continuity of solutions to the general dual-polar Orlicz-Minkowski problem for general nonzero finite Borel measure $\mu$ (instead of discrete measures) are provided. Our proofs are based on the approximation of convex bodies by polytopes. 

 Section \ref{Section:5} aims to investigate several variations of the general dual-polar Orlicz-Minkowski problem, including those leading to the most general definitions extending the $L_p$ Petty bodies (see Section \ref{section 5-3}). In Section \ref{section 5-1}, the objective functional  $\int_{\sphere}\varphi(h_{Q^\circ}(u))d\mu(u)$  in (\ref{dual-polar-4-12-19}) will be replaced by the ``Orlicz norm"  $\|h_{Q^\circ}\|_{\mu,\varphi}$. In this case,  the continuity, uniqueness, and existence of solutions are rather similar to those in Section \ref{section-4}. The second variation, considered in Section \ref{section 5-2}, is quite different from the general dual-polar Orlicz-Minkowski problem (\ref{dual-polar-4-12-19}). It replaces the general dual volume $\dveV(\cdot)$ by the general volume formulated as follows: for $K\in \cKon$,  \begin{equation*}
\vG(K)=\int_{\sphere}G(h_{K}(u),u)\,dS_K(u),
\end{equation*} where $S_K$ denotes the surface area measure of $K$ defined on $\sphere$. Although the geometric invariant $\vG(\cdot)$  has most properties required to solve the related polar Orlicz-Minkowski problem, it lacks the monotonicity in terms of set inclusion, a key ingredient in the proofs of the main results in Section   \ref{section-4}. With the help of the celebrated isoperimetric inequality, we are able to find a substitution of Lemma \ref{lemmaforinterior} for $\vG(\cdot)$ and this will be stated in Lemma \ref{lemmaforinterior-polar}.  Consequently,  the existence of solutions to the related polar Orlicz-Minkowski problem is established in Theorem \ref{polardualminkowskvG-2-2-8}.

\section{Preliminaries and Notations}\label{Section-2}

In the $n$-dimensional Euclidean space $\Rn$, $\ball$ denotes the
unit Euclidean ball and $\sphere$ denotes the unit sphere. Denote by
$\{e_1, \cdots,  e_n\}$ the canonical orthonormal basis of $\Rn$. By
$\cKn$  we mean the set of all compact convex subsets of $\Rn$. For
each $K\in \cKn$, one can define its support function
$h_{K}:\sphere\rightarrow \R$ by $h_{K}(u)=\max_{x\in K}\langle
x,u\rangle$ for any $u\in\sphere,$ where $\langle x, y\rangle$ is
the usual inner product in $\Rn$.  A natural metric on $\cKn$ is the
Hausdorff metric $d_H$, where for $K, L\in \cKn$, one has $$d_H(K,
L)=\|h_{K}-h_{L}\|_{\infty}=\max_{u\in\sphere}|h_K(u)-h_L(u)|.$$ We
say the sequence $K_1, K_2, \cdots, K_i, \cdots \in \cKn$ converges
to $K\in \cKn$ in the Hausdorff metric, denoted by $K_i\rightarrow
K$,  if $\lim_{i\rightarrow\infty}d_H(K_i,  K)=0.$ The Blaschke
selection theorem provides a powerful machinery to solve Minkowski
type problems. It asserts that  if $K_i\in \cKn$ and  there  exists a constant $R$ such that $K_i\subset R\ball$ for all $i\in
\N$,
then there exist a subsequence $\{K_{i_j}\}_{j\geq 1}$ of  $\{K_{i}\}_{i\geq 1}$ and $K\in \cKn$ such that
$K_{i_j}\rightarrow K$ as $j\rightarrow \infty$ in the Hausdorff
metric.

Denote by $o$ the origin of $\Rn$.  A convex body in $\Rn$ is a
compact convex subset of $\R^n$ with nonempty interior. Let
$\cKo\subset \cKn$ denote the set of all  convex bodies containing
$o$. For $K\in \cKo$, $h_K$ is a nonnegative function defined on
$\sphere$.  Besides the support function, for $K\in \cKo$, one can
also define the radial function
$\rho_{K}:\sphere\rightarrow[0,\infty)$ by
$\rho_{K}(u)=\max\{\lambda\ge 0:\lambda u\in K\}$ for $u \in
\sphere$. In particular, $\rho_K(u)u\in \partial K$, where $\partial
K$ denotes the boundary of $K$. For convenience, in later context,
we will also use $\into K$ to denote the interior of $K$. It can be
easily checked that $\rho_{sK}=s\cdot \rho_K$ and $h_{sK}=s\cdot
h_K$ for $s>0$ and $K\in \cKo$.

 Associated to each $K\in \cKo$ is the surface area measure
$S_K(\cdot)$ defined on $\sphere$ which may be formulated by
$S_K(\eta)=\cH(\nu_K^{-1}(\eta))$ for each Borel set $\eta\subset
\sphere$ (see e.g., \cite{Sch}), where  $\cH$ is the $(n-1)$
dimensional Hausdorff measure of $\partial K$, $\nu_K$ denotes the Gauss map of $K$
and $\nu_K^{-1}$ denotes the reverse Gauss map of $K$.   It is
worthwhile to mention that for  $K\in\cKo$, its volume, denoted by
$V(K)$,  takes the following forms:
\[
V(K)=\frac{1}{n}\int_{\sphere}h_K(u)\,dS_K(u)=\frac{1}{n}\int_{\sphere}\rho_K(u)^n\,du,
\]  where $\,du$ denotes the spherical measure of $\sphere$ (i.e., the Hausdorff measure on $\sphere$).

Let $\cKon\subset\cKo$ be the set of  convex bodies in $\Rn$ with the origin $o$ in their interiors. For each $K\in \cKon$, both $h_K$ and $\rho_K$ are strictly positive functions on $\sphere$. A useful fact is that $K_i\rightarrow K$,  with $K_i\in \cKon$   for all $i\in \N$ and $K\in \cKon$, in the Hausdorff metric is equivalent to $\rho_{K_i}$ convergent to $\rho_K$ uniformly on $\sphere$.  The polar body of $K\in\cKon$, denoted by $K^\circ$, may be formulated by
    \[
    K^\circ=\big\{x\in\R^n:~\langle x,y\rangle\leq 1 \text{ for any } y\in K\big\}.
    \]  An easily established fact is that if $K\in\cKon$, then
    $K^\circ\in\cKon$ and $K=K^{\circ\circ}.$ Moreover,  $
    \rho_{K^\circ}(u)\cdot h_{K}(u)=1$  for any $K\in\cKon$ and for any $u\in \sphere$. Clearly,
$S_{tK}=t^{n-1}S_K$ for any $t>0$ and $K\in \cKon$.
  For more background in convex geometry, please see e.g., \cite{Gruber2007, Sch}.

Let $G:\ (0,\infty)\times\sphere\rightarrow(0,\infty)$ be a
continuous function. The general dual volume of $K\in
\cKon$, denoted by $\dveV(K)$, was proposed in \cite{GHWXY} as follows:
\begin{equation}\label{general-dual-2019-1-12}
\dveV(K)=\int_{S^{n-1}}G(\rho_K(u),u)\,du.\end{equation}  When $G:
[0, \infty)\times \sphere\rightarrow [0, \infty)$,  the general dual
volume can be defined for $K\in \cKo$ with the formula same as
(\ref{general-dual-2019-1-12}).  Note that the general dual volume
$\dveV(\cdot)$ was used to derive the general dual Orlicz curvature
measures and hence plays central roles in establishing the existence
of solutions to the recently proposed general dual Orlicz-Minkowski
problem \cite{GHWXY, GHXY}. When $G(t, u)=\frac{1}{n}t^n$, one gets
$ \dveV(K) =V(K)$, and when $G(t, u)=\frac{1}{n}t^q$ for $q\neq 0,
n$,  $ \dveV(K)$ becomes the $q$th dual volume $\widetilde{V}_q(K)$
which plays fundamental roles in the dual Brunn-Minkowski theory
\cite{L1, Lut1988, Lutwak1990} and the $L_p$ dual Minkowski problem
(see e.g., \cite{BorFor, BHP, BorLYZZ, ChenHuangZhao,  LYZActa,
HuangZhao,  LYZ-Lp, zhao}). When $G(t, u)=G(t, e_1)$ for all $(t,
u)\in (0, \infty)\times \sphere$, $\dveV(K)$ becomes  the dual
Orlicz-quermassintegral  in \cite{ZSY2017}; while if $G(t,
u)=\int_0^t\phi(ru)r^{n-1}\,dr$ or $G(t,
u)=\int_t^{\infty}\phi(ru)r^{n-1}\,dr$ for some function $\phi:
\Rn\rightarrow (0, \infty)$, then  $\dveV(K)$ becomes the general
dual Orlicz quermassintegral  in \cite{XY2017-1}.  See \cite{GHWXY}
for more special cases. It has been proved that
$\dveV(K_i)\rightarrow \dveV(K)$  for $G: (0, \infty) \times \sphere \rightarrow (0,
\infty)$ being continuous and $K_i\rightarrow K$ with $K, K_i\in \cKon$ for all $i\in \N$ \cite[Lemma
6.1]{GHWXY} or  $G: [0, \infty)\times \sphere \rightarrow [0, \infty)$ being continuous and
$K_i\rightarrow K$ with $K, K_i\in \cKo$ for all $i\in \N$ \cite[Lemma 3.2]{GHXY}.  It
is easy to check that $\dveV(\cdot)$ in general is not homogeneous
on $\cKo$ and/or $\cKon$. Note that the general dual volume
$\dveV(\cdot)$ can be defined not only for convex bodies, but also
for star-shaped sets, see \cite{GHWXY} for more details.

The following property may be useful in later context. Denote by
$O(n)$ the set of all orthogonal matrices on $\Rn$, that is, for any
$T\in O(n)$, one has $TT^t=T^tT=\mathbb{I}_n$, where $T^t$ denotes
the transpose of $T$ and $\mathbb{I}_n$ is the identity map on
$\Rn$.
 \bp \label{rotation invariance}
Let $K\in \cKon$. If  $G(t, u)=\phi(t)$ for
all $(t, u)\in (0, \infty)\times\sphere$ with  $\phi: (0, \infty)\rightarrow
(0, \infty)$ being a continuous function, then  $\dveV(TK)=\dveV(K).$
\begin{proof} Let $G(t, u)=\phi(t)$ for all $t>0$ and $u\in \sphere$.  For $K\in \cKon$ and $T\in O(n)$, then the  determinant of $T$ is $\pm 1$  and \begin{eqnarray*}
\dveV(TK)=\int_{S^{n-1}}\phi(\rho_{TK}(u))\,du =\int_{S^{n-1}}\phi(\rho_{K}(T^{t}u))\,du=\int_{S^{n-1}}\phi(\rho_{K}(v))\,dv
        =\dveV(K),
 \end{eqnarray*} if letting $T^tu=v$. This completes the proof.
 \end{proof}  \ep
 In later context, we will employ Proposition \ref{rotation
invariance} to $G(t, u)=\frac{1}{n}t^q$ for $0\neq
q\in \R$, which implies  $\widetilde{V}_q(TK)=\widetilde{V}_q(K)$  for all
$T\in O(n)$ and all $K\in \cKon$.

The following result is an easy consequence of the weak convergence of
$\mu_i\rightarrow\mu$, but plays essential roles in our later
context. Its proof is simple and will be omitted.

\begin{lemma}\label{uniformly converge-lemma} Let $\mu, \mu_i$ for each $i\in \N$ be nonzero finite Borel measures on $\sphere$ such that $\mu_i\rightarrow \mu$ weakly. Let  $f, f_i$ for each $i\in\N$ be continuous functions on $\sphere$ such that $f_i\rightarrow f$ uniformly on $\sphere$. Then,
$$\lim_{i\rightarrow \infty}\int_{S^{n-1}}f_i\,d\mu_i=\int_{S^{n-1}}f\,d\mu.$$
\end{lemma}

\section{The homogeneous general dual  volumes and properties}\label{Section:3}
Throughout this paper, $G:\
(0,\infty)\times\sphere\rightarrow(0,\infty)$ is always assumed to
be continuous.  In this section, we will define the homogeneous
general dual  volume and discuss related properties.  For
simplicity, let
\begin{eqnarray*}
\cG_I\!\!&=&\! \!\Big\{G:\ G(t, \cdot) \ \text{is\ continuous,\ strictly\ increasing\ on $t$},\  \lim_{t\rightarrow 0^+}G(t, \cdot)=0,~\lim_{t\rightarrow\infty}G(t, \cdot)=\infty \Big \},\\
\cG_d\!\!&=&\!\! \Big\{G:\ G(t, \cdot) \ \text{is\ continuous,\ strictly\ decreasing\ on $t$},\  \lim_{t\rightarrow 0^+}G(t, \cdot)=\infty,~\lim_{t\rightarrow\infty}G(t, \cdot)=0 \Big \}.
\end{eqnarray*}  The homogeneous general dual volume of  $K\in\cKon$, denoted by $\hG(K)$,  can be formulated by
    \begin{eqnarray}\label{homogenerousvolumeincreasing}
    \hG(K)&=&\inf\left\{\eta>0:\ \   \int_{\sphere}G\bigg(\frac{\rho_{K}(u)}{\eta},u\bigg)
    \,du\leq 1\right\},\ \ \ ~ \text{if}~G\in\cG_I, \\  \label{homogenerousvolumedecreasing}
    \hG(K)&=&\inf\left\{\eta>0:  \ \ \int_{\sphere}G\bigg(\frac{\rho_{K}(u)}{\eta},u\bigg)
    \,du\geq 1\right\},\ \ \ ~ \text{if}~ G\in\cG_d.
    \end{eqnarray}

The following proposition provides a more convenient formula for $\hG(\cdot)$.
 \begin{proposition} \label{equivalent-form-1-23}
Let $K\in\cKon$. For any $G\in \cG_I\cup \cG_d$, there exists a unique $\eta_0>0$ such that
\begin{equation}\label{homogenerousnormforG}
 \int_{\sphere}G\bigg(\frac{\rho_{K}(u)}{\eta_0},u\bigg)\,du=1.
\end{equation} Moreover, $\eta_0=\hG(K)$.
\end{proposition}
\begin{proof} The proof of this result is standard. For  $\eta\in (0, \infty)$ and $K\in\cKon$, let $G\in \cG_I$ and
$$
H_K(\eta)=  \int_{\sphere}G\bigg(\frac{\rho_K(u)}{\eta},u\bigg)\,du.
$$ As  $K\in\cKon$, there exist positive constants $r$ and $R$ such that $r\le \rho_K\le R$.
Thus for any $u\in S^{n-1}$,
 \begin{equation}\label{equivalent definition-ineq-1}
 \int_{S^{n-1}}G\bigg(\frac{r}{\eta},u\bigg)\,du\ \le \
H_K(\eta)\ \le \
 \int_{S^{n-1}}G\bigg(\frac{R}{\eta},u\bigg)\,du.
 \end{equation}
This, together with $G\in\cG_I$ and Fatou's lemma, implies that
 $$\liminf_{\eta\rightarrow 0^+}H_K(\eta)\ \ge \ \liminf_{\eta\rightarrow 0^+}
 \int_{\sphere}G\left(\frac{r}{\eta},u\right)du\ \ge \
 \int_{\sphere}\liminf_{\eta\rightarrow
0^+}G\left(\frac{r}{\eta},u\right)du=\infty.
$$
 On the other hand, the dominated convergence theorem yields, by (\ref{equivalent definition-ineq-1}), that
 $$
\lim_{\eta\rightarrow \infty}H_K(\eta)\ \leq\  \lim_{\eta\rightarrow \infty} \int_{S^{n-1}}G\bigg(\frac{R}{\eta},u \bigg)\,du= \int_{S^{n-1}} \lim_{\eta\rightarrow \infty}G\bigg(\frac{R}{\eta},u\bigg)\,du=0.
$$ Thus, $\lim_{\eta\rightarrow 0^+}H_K(\eta)=\infty$   and
 $\lim_{\eta\rightarrow \infty}H_K(\eta)=0.$ As $G\in \cG_I$ is continuous and strictly increasing,
 $H_K(\eta)$ is clearly continuous and strictly decreasing on $\eta\in (0, \infty)$. Hence, there exists a unique $\eta_0>0$ such that $H_K(\eta_0)=1$, which proves
(\ref{homogenerousnormforG}). Clearly $\eta_0=\hG(K)$ by
(\ref{homogenerousvolumeincreasing}).

The case for $G\in \cG_d$ follows along the similar lines as above, and its proof will be omitted.   \end{proof}

Clearly, if $G(t, u)=t^q/n$ with $q\neq 0$ for all $(t, u)\in (0, \infty)\times \sphere$, then $$\hG(K)=\bigg(\frac{1}{n}\int_{\sphere} \rho_K^q(u)\,du\bigg)^{1/q}= \big(\widetilde{V}_q(K)\big)^{1/q}.$$  Properties for $\hG(\cdot)$ are summarized in the following proposition.

\bp
\label{properties-G-hat}
Let $G\in \cG_I\cup \cG_d$. Then $\hG(\cdot)$ has the following properties.

\vskip 1mm \noindent i) $\hG(\cdot)$ is homogeneous, that is,
$\hG(sK)=s\hG(K)$ holds for all $s> 0$ and all $K\in \cKon$.

\vskip 1mm \noindent ii)  $\hG(\cdot)$ is continuous on $\cKon$ in terms of the Hausdorff metric, that is, for any sequence $\{K_i\}_{i\geq 1}$ such that $K_i\in \cKon$ for all $i\in \N$ and $K_i\rightarrow K\in \cKon$, then  $\hG(K_i)\rightarrow \hG(K). $

\vskip 1mm\noindent iii) $\hG(\cdot)$ is strictly increasing, that is, for any $K, L\in \cKon$ such that
$K\subsetneq L$, then $\hG(K)<\hG(L)$. \ep

 \begin{proof}  i) The desired argument follows trivially from Proposition \ref{equivalent-form-1-23},  and  $\rho_{sK}=s\rho_K$ for all $s>0$.

  \vskip 2mm \noindent ii) Let  $K_i\in \cKon$ for all $i\in \N$ and $K_i\rightarrow K\in \cKon$.   Then  $\rho_{K_i}\rightarrow \rho_K$ uniformly on $\sphere$. Moreover, there exist two positive constants $r_K<R_K$ such that $r_K\leq \rho_K\leq R_K$ and $r_K\leq \rho_{K_i}\leq R_K$ for all $ i\in \N. $ For $G\in \cG_I$, it follows from  Proposition \ref{equivalent-form-1-23} and (\ref{equivalent definition-ineq-1})
that for each $i\in \N$,  \begin{eqnarray*} \label{compare-1-23-1}
\int_{\sphere}G\bigg(\frac{r_K}{\hG(K_i)},u\bigg)\,du\ \leq \ 1=\ \int_{\sphere}G\bigg(\frac{\rho_{K_i}(u)}{\hG(K_i)},u\bigg)\,du\
\leq  \int_{\sphere}G\bigg(\frac{R_K}{\hG(K_i)},u\bigg)\,du.
\end{eqnarray*}  Suppose that $\inf_{i\in\mathbb{N}}
\widehat{V}_G(K_i)=0$, and without loss of generality, assume that  $\lim_{i\rightarrow \infty}
\widehat{V}_G(K_i)=0$. Then for any $\varepsilon>0$, there exists $i_{\varepsilon}\in \N$ such that $\hG(K_i)<\varepsilon$ for all $i>i_{\varepsilon}$. Thus, for $i> i_{\varepsilon}$, \begin{eqnarray*}
 \int_{\sphere}G\bigg(\frac{r_K}{\varepsilon},u\bigg)\,du\ \le \ \int_{\sphere}G\bigg(\frac{r_K}{\hG(K_i)},u\bigg)\,du\ \leq \ 1.
\end{eqnarray*}  Fatou's lemma and the fact that $\lim_{t\rightarrow
\infty}G(t, \cdot)=\infty$ yield
\begin{eqnarray*}
 \infty=\int_{\sphere}\liminf_{\varepsilon\rightarrow 0^+}G\bigg(\frac{r_K}{\varepsilon},u\bigg)\,du\ \le \ \liminf_{\varepsilon\rightarrow 0^+}\int_{\sphere}G\bigg(\frac{r_K}{\varepsilon},u\bigg)\,du\  \leq \ 1,
\end{eqnarray*} a contradiction. Hence, $A_1=\inf_{i\in\mathbb{N}}
\widehat{V}_G(K_i)>0.$  Moreover, for all $u\in \sphere$ and all $i\in \N$,  $$G\bigg(\frac{\rho_{K_i}(u)}{\hG(K_i)},u\bigg)\leq G\bigg(\frac{R_{K}}{A_1},u\bigg). $$

Assume that $\limsup_{i\rightarrow \infty} \hG(K_i)>\hG(K)$. There exists a subsequence $\{K_{i_j}\}$ of $\{K_i\}$ such that $\lim_{j\rightarrow \infty} \hG(K_{i_j})>\hG(K).$ Together with Proposition \ref{equivalent-form-1-23}  and the dominated convergence theorem, one has \begin{eqnarray*}
 1&=& \lim_{j\rightarrow \infty} \int_{\sphere}G\bigg(\frac{\rho_{K_{i_j}}(u)}{\hG(K_{i_j})},u\bigg)\,du\\ &=& \int_{\sphere} \lim_{j\rightarrow \infty} G\bigg(\frac{\rho_{K_{i_j}}(u)}{\hG(K_{i_j})},u\bigg)\,du\\ &=&\int_{\sphere} G\bigg(\frac{\rho_{K}(u)}{ \lim_{j\rightarrow \infty}\hG(K_{i_j})},u\bigg)\,du \\ &<& \int_{\sphere} G\bigg(\frac{\rho_{K}(u)}{\hG(K)},u\bigg)\,du=1.
\end{eqnarray*}
This is a contradiction and hence $\limsup_{i\rightarrow \infty}
\hG(K_i) \leq \hG(K).$ Similarly, $\liminf_{i\rightarrow \infty}
\hG(K_i) \geq \hG(K)$  also holds, which leads to
$\lim_{i\rightarrow \infty} \hG(K_i) =\hG(K)$ as desired.

The case for $G\in \cG_d$ follows along the same lines, and its proof will be omitted.

\vskip 2mm \noindent iii) Let $G\in \cG_I$ and let $K, L\in \cKon$ such that $K\subsetneq L$. Then, the
spherical measure of the set $E=\{u\in \sphere:
\rho_K(u)<\rho_L(u)\}$ is positive.  By Proposition
\ref{equivalent-form-1-23}, one has \begin{eqnarray*}1 &=&\int_{\sphere}G\bigg(\frac{\rho_{L}(u)}{\hG(L)},u\bigg)\,du\\&=&
\int_{\sphere}G\bigg(\frac{\rho_{K}(u)}{\hG(K)},u\bigg)\,du\\ &=&
\int_{E}G\bigg(\frac{\rho_{K}(u)}{\hG(K)},u\bigg)\,du+\int_{\sphere\setminus
E}G\bigg(\frac{\rho_{K}(u)}{\hG(K)},u\bigg)\,du\\ &<&
\int_{E}G\bigg(\frac{\rho_{L}(u)}{\hG(K)},u\bigg)\,du+\int_{\sphere\setminus
E}G\bigg(\frac{\rho_{L}(u)}{\hG(K)},u\bigg)\,du\\ &=&
\int_{\sphere}G\bigg(\frac{\rho_{L}(u)}{\hG(K)},u\bigg)\,du.
\end{eqnarray*}  Then $\hG(K)<\hG(L)$ follows from the fact that $G(t, \cdot)$ is strictly increasing on $t\in (0, \infty)$.

The case for $G\in \cG_d$ follows along the same lines, and its proof will be omitted.
 \end{proof}

 For $G: (0, \infty) \times \sphere \rightarrow (0, \infty)$, define two families of convex bodies as follows:
 \begin{eqnarray*}
 \cBt&=&\big\{Q\in\cKon:\  \dveV{(Q^\circ)}=\dveV{(\ball)}\big\}; \label{def-B-1-23-1}\\
\cBh&=&\big\{Q\in\cK_{(o)}^n:\
\hG(Q^\circ)=\hG(\ball)\big\},  \ \ \ \text{if} \ \ G\in \cG_I\cup \cG_d.\nonumber
 \end{eqnarray*}
 It is obvious that both $\cBt$ and $\cBh$ are nonempty as they all contain the unit Euclidean ball
$\ball$.

The following lemma plays essential roles in later context.

\begin{lemma}\label{lemmaforinterior} Let
$G: (0,\infty)\times\sphere\rightarrow (0, \infty)$ be a continuous function. For $q\in \R$, let $G_q(t, u)= \frac{G(t,u)}{t^q}$. Suppose that  there exists a constant   $q\geq n-1$ such that
    \begin{equation}\label{conditionforG}
    \inf\Big\{G_q(t,u): \ \ \ t\geq 1 \ \ \mathrm{and}\  \ u\in \sphere\Big\} >0.
    \end{equation}  Then the following statements hold.

 \vskip 2mm \noindent i) If $\{Q_i\}_{i\geq 1}$ with $Q_i\in \cBt$ for all $i\in \N$ is a bounded sequence, then there exist a subsequence $\{Q_{i_j}\}_{j\geq 1}$ of $\{Q_i\}_{i\geq 1}$ and a convex body $Q_0\in \cBt$ such that $Q_{i_j}\rightarrow Q_0.$

\vskip 2mm \noindent ii) If in addition $G\in \cG_I $, the statement in i) also holds if $\cBt$ is replaced by $\cBh$.
 \end{lemma}

  \noindent {\bf Remark.} Clearly $G(t, u)=t^q$ for some $q\geq n-1$ satisfies  (\ref{conditionforG}). In particular $G(t, u)=t^n/n$ satisfies  (\ref{conditionforG}) and hence Lemma \ref{lemmaforinterior} recovers \cite[Lemma 3.2]{Lutwak1996}.  It can be easily checked that formula (\ref{conditionforG}) is equivalent to: there exist constants $c, C>0$, such that  \begin{equation}\label{conditionforG--1}
    \inf\Big\{G_q(t,u): \ \ \ t\geq c \ \ \mathrm{and}\  \ u\in \sphere\Big\} >C.
    \end{equation} Moreover,  if $G\in \cG_d$, then $G$ does not satisfy (\ref{conditionforG}). In fact, for all $q\geq n-1$ and for all $u\in \sphere$,  $$\lim_{t\rightarrow\infty} G_q(t, u)= \lim_{t\rightarrow\infty} G(t, u)\times \lim_{t\rightarrow\infty} t^{-q} =0. $$

 \begin{proof} Let $\{Q_i\}_{i\geq 1}$  be a bounded sequence with $Q_i\in \cBt$ (or, respectively, $Q_i\in \cBh$) for all $i\in \N$. It follows from the Blaschke selection theorem that there exist a subsequence of  $\{Q_i\}_{i\geq 1}$, say $\{Q_{i_j}\}_{j\geq 1}$, and  a compact convex set $Q_0\in \cKn$, such that $Q_{i_j}\rightarrow Q_0$ in the Hausdorff metric. As $o\in \mathrm{int} Q_{i_j}$ for all $j\in \N$, one has, $o\in Q_0$. In order to show $Q_0\in \cBt$ (or, respectively, $Q_0\in \cBh$),  we first need to show $o\in \mathrm{int}Q_0$.

 \vskip 2mm \noindent i) To this end, we assume that $o\in \partial Q_0$ and seek for contradictions. As $\{Q_i\}_{i\geq 1}$ is a bounded sequence, there exists a constant $R>0$ such that $Q_i\subset R\ball$ for each $i\in \N$. For each $j\in \N$, one can find $u_{i_j}\in\sphere$ such that
$r_{i_j}=h_{Q_{i_j}}(u_{i_j})=\min_{u\in\sphere}h_{Q_{i_j}}(u).$  As $o\in \partial Q_0$, one sees that  $\lim_{j\rightarrow\infty}r_{i_j}=0$. The fact that $Q_{i_j}\subset R \ball$  implies that $\frac{1}{R}  \ball \subset  Q_{i_j}^\circ$, and in particular,  $
    \rho_{Q_{i_j}^\circ}(u)\geq\frac{1}{R}$ for any $u\in\sphere$.

   Let the constant $c$ in (\ref{conditionforG--1}) be $\frac{1}{R}$. For some fixed constants $q\geq n-1$ and $C>0$,
    \begin{eqnarray}\label{boundnessfordvev}
    \dveV(Q_{i_j}^\circ)=\int_{\sphere}G\big(\rho_{Q_{i_j}^\circ}(u),u\big)\,du\geq C\int_{\sphere}\big(\rho_{Q_{i_j}^\circ}(u)\big)^q\,du=Cn \widetilde{V}_q(Q_{i_j}^\circ).
    \end{eqnarray}
    For any $T\in O(n)$,
$(TQ_{i_j})^\circ=(T^t)^{-1}Q_{i_j}^\circ=TQ_{i_j}^\circ$ as
$T^tT=\mathbb{I}_n$ where $T^{-1}$ denotes the inverse map of $T$.
It follows from Proposition \ref{rotation invariance} that
$\widetilde{V}_q(Q_{i_j}^\circ)$ is $O(n)$-invariant.  Hence, for
convenience, one can assume
that $u_{i_j}=e_n$. The radial function
$\rho_{Q_{i_j}^\circ}$ can be bounded from below by the radial
function of $\bC_j=Cone\Big(o, \frac{1}{R},
\frac{e_n}{r_{i_j}}\Big)$, the cone with base $\frac{B^{n-1}}{R}$
and  the apex $\frac{e_n}{r_{i_j}}$. Note that
$$\rho_{\bC_j}(u)=\left\{
\begin{array}{ll} \frac{1}{R \sin \theta+r_{i_j}\cos \theta},  & \
\text{if}\ \ u\in \sphere\ \ \text{such\ that}\ \ \langle e_n,
u\rangle\geq 0; \\ 0, &  \  \ \text{if}\ \ u\in \sphere\ \
\text{such\ that}\ \ \langle e_n, u\rangle< 0,\end{array} \right.$$
where $\theta\in [0, \pi/2]$ is the angle between $u$ and $e_n$ (see
Figure \ref{fig-1-1}).

    \begin{figure}[!htb]
        \center{\includegraphics[width=80mm]
        {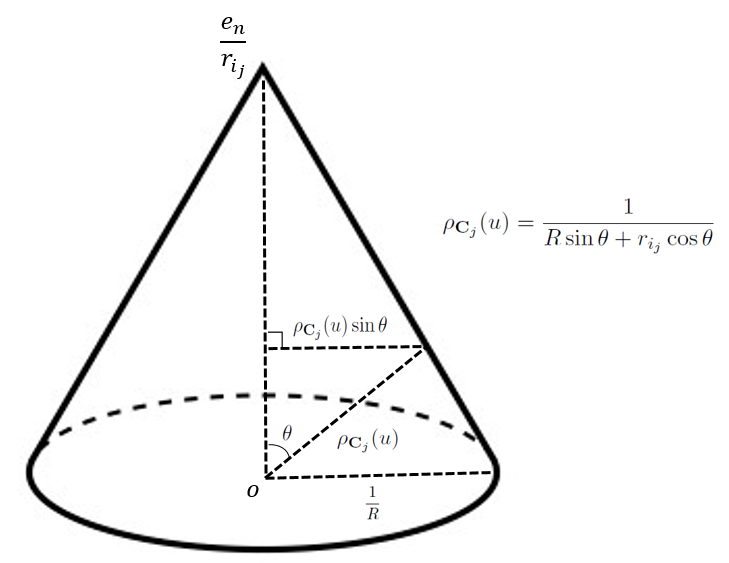}}
        \caption{\label{fig-1-1} The cone $\bC_j$}
      \end{figure}

\vskip 2mm Indeed, from Figure \ref{fig-1-1}, for $u\in \sphere$
such that $\langle u, e_n\rangle\geq 0$, one has
$$\frac{\rho_{\bC_j}(u)\cdot \sin
\theta}{R^{-1}}=\frac{r_{i_j}^{-1}-\rho_{\bC_j}(u)\cdot \cos
\theta}{r_{i_j}^{-1}}\Longrightarrow \rho_{\bC_j}(u)= \frac{1}{R
\sin \theta+r_{i_j}\cos \theta}.$$ Using the general spherical
coordinate (see, e.g., \cite[Page 14]{BorLYZZ}) by letting
    \[
    u=(v\sin\theta, \cos\theta)\in\sphere, \ v\in S^{n-2} \ \ \text{and} \ \ \theta\in[0,\pi],
    \] we have   $\,du =   (\sin\theta)^{n-2}\,d\theta  \, dv,$ where $\,dv$ denotes the spherical measure of $S^{n-2}$. Thus  \begin{eqnarray} n \widetilde{V}_q(\bC_j)&=&  \int_{S^{n-2}} \bigg(
 \int_{0}^{\frac{\pi}{2}}
    \left(\frac{1}{R\sin\theta+r_{i_j}\cos\theta}\right)^q(\sin\theta)^{n-2}\,d\theta \bigg) \,   dv\nonumber \\ &=& (n-1)V(B^{n-1})
 \int_{0}^{\frac{\pi}{2}}
    \left(\frac{1}{R\sin\theta+r_{i_j}\cos\theta}\right)^q(\sin\theta)^{n-2} \,d\theta.\label{qth-dual-cone} \end{eqnarray}  We will not need the precise value of $\widetilde{V}_q(\bC_j)$, however if $q=n$, formula (\ref{qth-dual-cone}) does lead to  $$\widetilde{V}_n(\bC_j)=V(\bC_j)=\frac{V(B^{n-1})}{n R^{n-1} r_{i_j}},$$ which coincides with the calculation provided in \cite[Lemma 3.2]{Lutwak1996}.

    Together with (\ref{boundnessfordvev}), $\rho_{Q_{i_j}^\circ}\geq \rho_{\bC_j}$, Fatou's lemma, and $\lim_{j\rightarrow\infty}r_{i_j}=0$, one has, if $q\geq n-1$, then  $n-2-q\leq -1$  and
    \begin{eqnarray}  \liminf_{j\rightarrow \infty} \dveV(Q_{i_j}^\circ) &\geq&  \liminf_{j\rightarrow \infty} Cn \widetilde{V}_q(\bC_j) \nonumber \\ &=&C\cdot (n-1)V(B^{n-1})
 \cdot  \liminf_{j\rightarrow \infty}  \int_{0}^{\frac{\pi}{2}}
    \left(\frac{1}{R\sin\theta+r_{i_j}\cos\theta}\right)^q(\sin\theta)^{n-2}d\theta \nonumber \\  &\geq&C\cdot (n-1)V(B^{n-1})
 \cdot \int_{0}^{\frac{\pi}{2}}  \liminf_{j\rightarrow \infty}
    \left(\frac{1}{R\sin\theta+r_{i_j}\cos\theta}\right)^q(\sin\theta)^{n-2}\,d\theta \nonumber \\
&=&\frac{C\cdot (n-1)V(B^{n-1})}{R^{q}}\int_{0}^{\frac{\pi}{2}}(\sin\theta)^{n-2-q}\,d\theta\nonumber \\
&\geq &\frac{C\cdot (n-1)V(B^{n-1})}{R^{q}}\int_{0}^{\frac{\pi}{2}}\frac{1}{\sin\theta}\,d\theta\nonumber  \\
&=& \frac{C\cdot (n-1)V(B^{n-1})}{R^{q}} \cdot \ln\tan(\theta/2)
\Big|_{\theta=0}^{\theta=\pi/2} =\infty. \label{contradiction--1}
    \end{eqnarray}  On the other hand,  as $Q_{i_j}\in \cBt$ for each $j\in \N$, then
    $$\dveV(Q_{i_j}^\circ)=\dveV{(\ball)}=\int_{\sphere}G(1, u)\,du<\infty.$$
   This is a contradiction and thus $o\in\text{int}{Q_0}$.

   As $Q_{i_j}\in \cKon$ for each $j\in \N$ and $Q_0\in \cKon$, $Q_{i_j}\rightarrow Q_0$ yields $Q_{i_j}^{\circ} \rightarrow Q_0^{\circ}.$ Together with the continuity of $\dveV(\cdot)$  (see \cite[Lemma 6.1]{GHWXY}) and the fact that $\dveV(Q_{i_j}^\circ)=\dveV{(\ball)}$ for each $j\in \N$, one gets $\dveV(Q_0^\circ)=\lim_{j\rightarrow \infty} \dveV(Q_{i_j}^\circ)=\dveV{(\ball)}.$ This concludes that $Q_0\in \cBt$ as desired.

   \vskip 2mm \noindent ii)  Again, we assume that $o\in \partial Q_0$ and seek for contradictions.  It follows from  Proposition \ref{equivalent-form-1-23} that  $\hG(\ball)>0$ is a finite constant.  Following notations in i), Proposition \ref{equivalent-form-1-23}   and $\hG(Q^\circ_{i_j})=\hG(\ball)$ for each $j\in \N$  yield that   \begin{equation}  \int_{\sphere}G\bigg(\frac{\rho_{Q_{i_j}^\circ}(u)}{\hG(\ball)},u\bigg)\,du=1.\label{Q-i-j--1}
\end{equation}
As $\frac{1}{R}  \ball \subset  Q_{i_j}^\circ$ for each $j\in \N$,
one can take the constant $c$ in  (\ref{conditionforG--1}) to be
$\frac{1}{R\cdot \hG(\ball)}$ and there exists a constant $C>0$ such
that, for all $u\in \sphere$ and some $q\geq n-1$,
  \begin{equation*}  G\bigg(\frac{\rho_{Q_{i_j}^\circ}(u)}{\hG(\ball)},u\bigg)\geq  C\cdot \bigg(\frac{\rho_{Q_{i_j}^\circ}(u)}{\hG(\ball)}\bigg)^q.
\end{equation*}  Together with  (\ref{Q-i-j--1}), one has,
  \begin{equation*}  \int_{\sphere}C\cdot \bigg(\frac{\rho_{Q_{i_j}^\circ}(u)}{\hG(\ball)}\bigg)^q\,du\leq 1 \Longrightarrow  C\cdot  \int_{\sphere} \big({\rho_{Q_{i_j}^\circ}(u)}\big)^q\,du\leq \big({\hG(\ball)}\big)^q.
\end{equation*}
Similar to (\ref{contradiction--1}), one gets  \begin{equation*}\infty =\liminf_{j\rightarrow \infty} C\cdot  \int_{\sphere} \big({\rho_{Q_{i_j}^\circ}(u)}\big)^q\,du\leq \big({\hG(\ball)}\big)^q,
\end{equation*} a contradiction and hence $o\in\text{int}{Q_0}$. The rest of the proof follows along the lines in i), where the continuity of $\hG(\cdot)$
(see Proposition \ref{properties-G-hat}) shall be used.
\end{proof}

\section{The general dual-polar Orlicz-Minkowski problem} \label{section-4}

Motivated by the polar Orlicz-Minkowski problem proposed in \cite{LuoYeZhu} and by the general dual Orlicz-Minkowski problem proposed in \cite{GHWXY, GHXY}, we propose the following general dual-polar Orlicz-Minkowski problem:

 \begin{problem}\label{general-dual-polar}   Under what conditions on a nonzero finite Borel measure $\mu$ defined
on $\sphere$, continuous functions $\varphi: (0, \infty)\rightarrow (0, \infty)$ and
$G\in \cG_I\cup \cG_d$  can we find a convex body $K\in\cKon$ solving the following optimization problems:
 \begin{eqnarray}\label{Problem-1-1-1-1}
&&\inf /\sup \left\{ \int_{\sphere}\varphi(h_{Q}(u))d\mu(u): Q\in \cBt \right\};\\
&&\inf /\sup \left\{\int_{\sphere}\varphi(h_{Q}(u))d\mu(u): Q\in \cBh \right\}.\label{Problem-1-1-2-2}
 \end{eqnarray} \end{problem}

 Although the function $G$ in the optimization problem (\ref{Problem-1-1-1-1}) can be any continuous function $G: (0, \infty)\times \sphere\rightarrow (0, \infty)$, to find its solutions, only those  $G\in \cG_I\cup \cG_d$
with monotonicity will be considered. One reason is that most $G$ of interest (such as $G(t, u)=t^q/n$ for $0\neq q\in \R$) are monotone. More importantly, without the monotonicity of $G$, the set $\cBt$ may contain only one convex body $\ball$ (for instance, if $G(1, u)<G(t, u)$ for all $(t, u)\in (0, \infty)\times \sphere$ such that $t\neq 1$).  In this case, the optimization problem (\ref{Problem-1-1-1-1}) becomes trivial.  Note that when $G(t, u)=t^n/n$, both $\dveV(\cdot)$ and (essentially) $\hG$ are volume, then Problem \ref{general-dual-polar} becomes the polar Orlicz-Minkowski problem posed in \cite{LuoYeZhu}.

 In later context, we always assume that $\varphi: (0, \infty)\rightarrow (0, \infty)$ is a continuous function.  For convenience, let $\varphi(0+)= \lim_{t\rightarrow0^+}\varphi(t)$ and $\varphi(\infty)=\lim_{t\rightarrow\infty}\varphi(t)$ provided the above limits exist (either finite or infinite).
We shall need the following classes of functions:
\begin{eqnarray*}
\cI&=&\{\varphi:  \varphi \ \text{is strictly increasing on}\ \
(0,\infty)\ \ \text{with}\ \
 \varphi(0+)=0,~\varphi(1)=1\ \text{and}\ \
 \varphi(\infty)=\infty\}; \\  \cD&=&\{\varphi:  \varphi \ \text{is strictly decreasing on}\
(0,\infty)\ \ \text{with}\ \
 \varphi(0+)=\infty,~\varphi(1)=1\  \text{and}\ \
\varphi(\infty)=0\}.
\end{eqnarray*} Note that the normalization value $\varphi(1)=1$ is mainly for technique convenience and $\varphi(1)$ can be modified to any positive numbers.

\subsection{The general dual-polar Orlicz-Minkowski problem for discrete measures}

In this subsection, we will solve the  general dual-polar
Orlicz-Minkowski problem for discrete measures. Throughout this subsection,  let $\mu$ be a discrete measure of the following form:
\begin{eqnarray} \mu=\sum_{i=1}^m\lambda_i\delta_{u_i},
\label{discrete--1--1}  \end{eqnarray}  where $\lambda_i>0$,
$\delta_{u_i}$ denotes the Dirac measure at $u_i$,  and $\{u_1,
\cdots, u_m\}$ is a subset of $\sphere$ which is not concentrated on
any closed hemisphere (clearly $m\geq n+1$).  It has been proved in \cite[Propositions 3.1
and 3.3]{LuoYeZhu} that the solutions to the polar Orlicz-Minkowski
problem for discrete measures must be polytopes, the convex hulls of
finite points in $\Rn$.  It is well-known that all convex bodies can
be approximated by polytopes, and hence  to study the Minkowski type
problems for discrete measures is very important and receives
extensive attention, see e.g., \cite{BorFor, BHZ2016, Cole2015,
GHXY, HongYeZhang-2017, Huanghe, HugLYZ, JianLu, liaijun2014,
WuXiLeng, GZhu2015I, GZhu2015II, GZhupreprint}.

The following lemma shows that if, when the infimum is considered,
Problem \ref{general-dual-polar}  for discrete measures has
solutions, then the solutions must be polytopes.

\begin{lemma}\label{polytope-lemma}
Let $\varphi\in\mathcal{I}$ and
$\mu$  be as in (\ref{discrete--1--1}) whose support $\{u_1,\cdots, u_m\}$
is not concentrated on any closed hemisphere. Let
$G\in \cG_I$.

\vskip 2mm \noindent i)
If $\widetilde{M}\in \cBt$ is a solution to the optimization problem (\ref{Problem-1-1-1-1}) when the infimum is considered,  then $\widetilde{M}$ is a polytope, and  $u_1,\cdots, u_m$ are the corresponding unit normal vectors of its faces.

\vskip 2mm \noindent ii)
If $\widehat{M}\in \cBh$ is a solution to the optimization problem (\ref{Problem-1-1-2-2}) when the infimum is considered,  then $\widehat{M}$ is a polytope, and  $u_1,\cdots, u_m$ are the corresponding unit normal vectors of its faces.
\end{lemma}

\begin{proof} Let $G\in \cG_I$.  For discrete measure $\mu$ and $Q\in \cKon$, one has
$$\int_{\sphere}\varphi(h_{Q}(u))d\mu(u)=\sum_{i=1}^m\varphi(h_{Q}(u_i))\mu(\{u_i\})=\sum_{i=1}^m \lambda_{i} \varphi(h_{Q}(u_i)). $$
 i) Let $\widetilde{M}\in \cBt$ be a solution to the optimization problem (\ref{Problem-1-1-1-1}). Define the polytope $P$  as follows:  $\widetilde{M}\subseteq P$, $h_P(u_i)=h_{\widetilde{M}}(u_i)$ for $1\leq i\leq m$, and $u_1,\cdots, u_m$ are the corresponding unit normal vectors of the  faces of $P$. As $\widetilde{M}\in \cBt$, one has $\dveV(\widetilde{M}^\circ)=\dveV(\ball)$ and $o\in \mathrm{int}\widetilde{M}$. Hence $P\in \cKon$ and $P^\circ \subseteq
\widetilde{M}^\circ$. Similar to the proof of Proposition
\ref{properties-G-hat} iii),  one can obtain that $\dveV(\cdot)$ for
$G\in \cG_I$ is strictly increasing in terms of set inclusion. In
particular, $\dveV(P^\circ) \leq
\dveV(\widetilde{M}^\circ)=\dveV(\ball).$ As
$\lim_{t\rightarrow\infty}G(t, \cdot)=\infty,$ there exists $t_0\geq
1$ such that $ \dveV(t_0P^\circ) =\dveV(\ball).$  That is, $P/t_0\in
\cBt$. Due to the minimality of $\widetilde{M}$ and the fact that
$\varphi\in \cI$ is strictly increasing,  one has \begin{eqnarray*}
\sum_{i=1}^m \lambda_{i} \varphi(h_{P}(u_i))=\sum_{i=1}^m
\lambda_{i} \varphi(h_{\widetilde{M}}(u_i))\leq \sum_{i=1}^m
\lambda_{i} \varphi(h_{P/t_0}(u_i))\leq \sum_{i=1}^m \lambda_{i}
\varphi(h_{P}(u_i)),\end{eqnarray*}  which yields  $t_0=1$. Then,
$\dveV(P^\circ) =\dveV(\ball)=\dveV(\widetilde{M}^\circ)$ and hence
$P=\widetilde{M}$ following from $\widetilde{M}\subseteq P$.

\vskip 2mm \noindent ii) Proposition \ref{properties-G-hat} iii) asserts that, if $G\in \cG_I$,  $\hG(K)<\hG(L)$ for all $K, L\in \cKon$ such that $K\subsetneq L$. The proof in this case then follows along the same lines as in i), and will be omitted.  \end{proof}

The following result is for the existence of solutions to  Problem \ref{general-dual-polar}  for discrete measures if the infimum is considered.

\begin{theorem}\label{polardualorlicz1}
Let $\varphi\in\mathcal{I}$ and
$\mu$  be as in (\ref{discrete--1--1}) whose support $\{u_1,\cdots, u_m\}$
is not concentrated on any closed hemisphere. Let
$G\in \cG_I$ be a continuous function such that (\ref{conditionforG}) holds for some $q\geq n-1$.  Then the following statements hold.

\vskip 2mm \noindent i) There exists a polytope
$\widetilde{P}\in\cBt$ with $u_1,\cdots, u_m$ being the corresponding unit normal vectors of its faces,
 such that,
    \begin{equation} \label{model-question-1}
  \sum_{i=1}^m  \lambda_i\varphi(h_{\widetilde{P}} (u_i))=\inf \Big\{
 \sum_{i=1}^m \lambda_i \varphi(h_{Q}(u_i)): \ \ \ Q\in \cBt \Big\}.
    \end{equation}
 \noindent ii)  There exists a polytope
$\widehat{P}\in\cBh$ with $u_1,\cdots, u_m$ being the corresponding unit normal vectors of its faces, such that,
    \[
  \sum_{i=1}^m  \lambda_i\varphi(h_{\widehat{P}} (u_i))=\inf \Big\{
 \sum_{i=1}^m \lambda_i \varphi(h_{Q}(u_i)): \ \ \ Q\in \cBh \Big\}.
    \]\end{theorem}

\begin{proof} By Lemma \ref{polytope-lemma},  to solve (\ref{model-question-1}), it will be enough to find a solution for the following problem:  \begin{equation} \label{model-question-2}
\widetilde{\alpha}=\inf \Big\{
 \sum_{i=1}^m \lambda_i \varphi(z_i): \ \ z\in \R_+^m \ \ \text{such\ that} \ \ P(z)\in \cBt \Big\},
    \end{equation}
where $z=(z_1, \cdots, z_m)\in \R_+^m$  means that each $z_i>0$ and
$$P(z)=\bigcap_{i=1}^m\Big\{x\in \R^n: \ \langle x, u_i\rangle \leq
z_i\Big\}\subset \cKon.$$  Clearly $h_{P(z)}(u_i)\leq z_i$ for all
$i=1, 2, \cdots, m.$

Let $P_1=P(1, \cdots, 1)$. Then $\ball\subsetneq P_1$ and hence $P^\circ_1\subsetneq \ball$. As $G\in \cG_I$ one has $\dveV(P^\circ_1)<\dveV(\ball).$ The facts that $G(t, \cdot)$ is strictly increasing on $t$ and $\lim_{t\rightarrow\infty}G(t, \cdot)=\infty$ imply the existence of $t_1>1$ such that $\dveV(t_1P^\circ_1)=\dveV(\ball).$ In other words, $P_1/t_1\in \cBt$ and then the infimum in (\ref{model-question-2}) is not taken over an empty set.  Moreover, due to $\varphi\in \cI$ (in particular, $\varphi$ is strictly increasing and $\varphi(1)=1$) and $1/t_1<1$, one has, \begin{equation*}  \widetilde{\alpha}\leq  \varphi( 1/t_1) \sum_{i=1}^m \lambda_i \leq \sum_{i=1}^m \lambda_i. \end{equation*}  This in turn implies that $z\in \R^m_+$ in (\ref{model-question-2}) can be restricted in a bounded set, for instance, \begin{equation} \label{bounded-12-21} z_i\leq \varphi^{-1}\Big(\frac{\lambda_1+\cdots+\lambda_m}{\min_{1\leq i\leq m} \lambda_i}\Big), \ \ \ \text{for\ all}\  i=1, 2, \cdots, m.\end{equation}
 Let $z^1, \cdots, z^j\cdots \in \R^m_+$ be the limiting sequence of (\ref{model-question-2}), that is,
 $$
\widetilde{\alpha}= \lim_{j\rightarrow \infty} \sum_{i=1}^m
\lambda_i \varphi(z^j_i) \ \ \text{and} \ \
\dveV(P^\circ(z^j))=\dveV(\ball)\ \ \ \text{for \ all}\ \ j\in \N.
$$
Due to (\ref{bounded-12-21}), without loss of generality, we can
assume that $z^j\rightarrow z^0$ for some $z^0\in \R^m$ and hence
$P(z^j)\rightarrow P(z^0)$ in the Hausdorff metric (see e.g.,
\cite{Sch}). Lemma \ref{lemmaforinterior} yields that $P(z^0)\in
\cBt$, i.e.,  $\dveV(P^\circ(z^0))=\dveV(\ball)$ and  $o\in
\mathrm{int}P(z^0)$. In particular,  $z_i^0>0$
for all $i=1, 2, \cdots, m$.

On the other hand, we claim that $h_{P(z^0)}(u_i)=z^0_i$ for all
$i=1, 2, \cdots, m$. To this end, assume not, then there exists
$i_0\in \{1, 2, \cdots, m\}$ such that
$h_{P(z^0)}(u_{i_0})<z^0_{i_0}.$ As $\varphi\in \cI$ is strictly
increasing and $\lambda_{i_0}>0$, one clearly has   $$
\widetilde{\alpha}= \sum_{i=1}^m \lambda_i \varphi(z^0_i)>
\sum_{i\in \{1, 2, \cdots, m\}\setminus\{i_0\}} \lambda_i
\varphi(z^0_i)+\lambda_{i_0} \varphi(h_{P(z^0)}(u_{i_0})).$$ This
contradicts with the minimality of $\widetilde{\alpha}$.

Let $\widetilde{P}=P(z^0)$. Then $\widetilde{P}\in \cBt$ solves (\ref{model-question-2}) and hence (\ref{model-question-1}).  This concludes the proof of i).

\vskip2mm \noindent ii) The proof is almost identical to the one for i), and will be omitted.
 \end{proof}

It has been proved in \cite{LuoYeZhu} that the existence of
solutions to Problem \ref{general-dual-polar} for discrete measures
in general is invalid  when $G(t, u)=t^n/n$, if the supremum is
considered for $\varphi \in \cI\cup \cD$, or the infimum is
considered for $\varphi \in \cD$. One can also prove similar
arguments for  Problem \ref{general-dual-polar} for discrete
measures with more general $G\in \cG_I$, but more delicate calculations are required.  We only state the
following result as an example.

\begin{proposition}\label{denial of other possibilities}
    Let $\mu$  be as in (\ref{discrete--1--1}) whose support $\{u_1,\cdots, u_m\}$
is not concentrated on any closed hemisphere.  Let $G\in \cG_I$ be such that (\ref{conditionforG})  holds for some $q\geq n-1$.

        \vskip 2mm \noindent i)   If $\varphi\in\cD$ and the first coordinates of $u_1,u_2,\cdots, u_m$ are all nonzero, then
        $$
\inf_{Q\in \cBh}\sum_{i=1}^m \lambda_i \varphi(h_Q(u_i)) =0.
       $$
    ii)    If $\varphi\in\mathcal{I}\cup\mathcal{D}$, then
        $$\sup_{Q\in \cBh} \sum_{i=1}^m \lambda_i \varphi(h_Q(u_i)) =\infty.$$
\end{proposition}

\begin{proof}
i) For $0<\epsilon<1$, let
$T_\epsilon=\text{diag}(1,1,\cdots,1,\epsilon)$ and $L_\epsilon=
T_\epsilon \ball.$ It can be checked that $$
\rho_{L_\epsilon}(w)=\big(w_1^2+w_2^2+\cdots+w_{n-1}^2+w_n^2/\epsilon^2\big)^{-1/2}$$
for all $w=(w_1, \cdots, w_n)\in \sphere$.    Thus
$\rho_{L_\epsilon}(w)$ is increasing on $\epsilon>0$ for each $w\in
\sphere$ and then $L_\epsilon$ is increasing in the sense of set
inclusion on $\epsilon>0$. In particular,  $L_\epsilon\subset \ball$
and $\ball\subset  L_\epsilon^\circ=T_\epsilon^{-1} \ball$.
Moreover, $L_\epsilon^\circ=T_\epsilon^{-1} \ball$ is decreasing in
the sense of set inclusion on $\epsilon>0$, and so is
$\hG({L_\epsilon^\circ})$ due to $G\in \cG_I$. By the homogeneity of
$\hG(\cdot)$, one has $\hG(f(\epsilon){L_\epsilon^\circ})=
\hG(\ball)$   if
$$f(\epsilon)=\frac{\hG(\ball)}{\hG({L_\epsilon^\circ})}. $$ We now claim that
$f(\epsilon)\rightarrow 0$, which is equivalent to prove  $\hG({L_\epsilon^\circ})\rightarrow \infty$ as
$\epsilon\rightarrow 0^+.$

To this end, it is enough to prove that
$\sup_{0<\epsilon<1}\hG({L_\epsilon^\circ})=\infty$. Assume that
$\sup_{0<\epsilon<1}\hG({L_\epsilon^\circ})=A_0<\infty$. By
$\ball\subset  L_\epsilon^\circ$ and  (\ref{conditionforG--1}) with
$c=1/A_0$,  there exists a constant $C_A>0$ such that
\begin{eqnarray*}
\int_{\sphere}G\bigg(\frac{\rho_{L_\epsilon^\circ}(u)}{A_0},u\bigg)\,du
\geq C_A
\int_{\sphere}\bigg(\frac{\rho_{L_\epsilon^\circ}(u)}{A_0}\bigg)^q\,du
\geq C_A \int_{\sphere}\bigg(\frac{\rho_{\bC_{\epsilon}}(u)}{A_0}\bigg)^q\,du,
\end{eqnarray*}
where $\bC_{\epsilon}\subset L_\epsilon^\circ$ is the cone with the base  $B^{n-1}$ and  the
apex ${\epsilon^{-1}}{e_n}$. It follows from Proposition \ref{equivalent-form-1-23}, (\ref{qth-dual-cone}), (\ref{contradiction--1}) and $q\geq n-1$ that
\begin{eqnarray*} 1&=& \liminf_{\epsilon\rightarrow 0^+}\int_{\sphere}G\bigg(\frac{\rho_{L_\epsilon^\circ}(u)}{\hG({L_\epsilon^\circ})},u\bigg)\,du\\ &\geq&  \liminf_{\epsilon\rightarrow 0^+} \int_{\sphere}G\bigg(\frac{\rho_{L_\epsilon^\circ}(u)}{A_0},u\bigg)\,du\\
&\geq& \frac{C_A}{A_0^q} \cdot  \liminf_{\epsilon\rightarrow 0^+} \int_{\sphere} \big(
\rho_{\bC_{\epsilon}}(u)\big)^q\,du \nonumber \\ &=& \frac{C_A
(n-1)V(B^{n-1})}{A_0^q} \cdot   \liminf_{\epsilon\rightarrow 0^+} \int_{0}^{\frac{\pi}{2}}
    \left(\frac{1}{\sin\theta+\epsilon\cos\theta}\right)^q(\sin\theta)^{n-2} \,d\theta\\ &=&\infty. \end{eqnarray*} This is a contradiction, which yields $\sup_{0<\epsilon<1}\hG({L_\epsilon^\circ})=\infty$ and then  $f(\epsilon)\rightarrow 0$ as $\epsilon\rightarrow 0^+$.

    Recall that $\hG(f(\epsilon){L_\epsilon^\circ})= \hG(\ball)$ and then $L_\epsilon/f(\epsilon)=T_\epsilon \ball/f(\epsilon)\in \cBh.$ It is assumed that $\alpha=\min_{1\leq i\leq m}\{|(u_i)_1|\}>0$, and hence for all $1\leq i\leq m$ (by letting $v_2=u_i$), $$
h_{L_\epsilon/f(\epsilon)}(u_i)=\max_{v_1\in L_\epsilon/f(\epsilon)}\langle v_1,u_i \rangle
=\max_{v_2\in \ball}\langle T_\epsilon v_2,u_i \rangle/  f(\epsilon)\geq \alpha^2 /f(\epsilon).
$$ The fact that $\varphi\in \cD$ is strictly decreasing yields
\begin{eqnarray*}
            \inf_{Q\in \cBh}\sum_{i=1}^m \lambda_i \varphi(h_Q(u_i))
            \leq  \sum^m_{i=1}\varphi\left(h_{L_\epsilon/f(\epsilon)}(u_i) \right)\cdot\mu(\{u_i\})
            \leq\varphi\left(\alpha^2/f(\epsilon) \right)\cdot\mu(S^{n-1}) \rightarrow 0,
\end{eqnarray*} where we have used $\lim_{\epsilon\rightarrow 0^+}f(\epsilon)= 0$ and $\lim_{t\rightarrow \infty} \varphi(t)=0$. This concludes the proof of i).

  \vskip 2mm \noindent ii) Note that $\mu(\{u_1\})>0$. For any $0<\epsilon<1$, let $\widetilde{L}_\epsilon= TT_{\epsilon} \ball,$
where $T\in O(n)$ is an orthogonal matrix such that  $T^t u_1=e_1$
(indeed, this can always be done by the Gram-Schmidt process). Again  $\widetilde{L}_\epsilon\subset \ball$ and hence $\ball \subset
\widetilde{L}_\epsilon^\circ$.
As in i), one can prove that
$$  f(\epsilon)=\frac{\hG(\ball)}{\hG({\widetilde{L}_\epsilon^\circ})} \rightarrow 0  \ \ \text{as}\ \
\epsilon\rightarrow 0^+.$$ Moreover,
$\hG(f(\epsilon){\widetilde{L}_\epsilon^\circ})= \hG(\ball)$ and
thus $\widetilde{L}_\epsilon/f(\epsilon)\in \cBh.$ One can check (by
letting $v_2=e_1$) that $$
h_{\widetilde{L}_\epsilon/f(\epsilon)}(u_1) =f(\epsilon)^{-1}
\max_{v_2\in \ball}\langle TT_{\epsilon} v_2,u_1 \rangle =
f(\epsilon)^{-1} \langle T^t u_1, \text{diag}(1, 1,
\cdots,1,\epsilon)\cdot e_1 \rangle =f(\epsilon)^{-1}.
$$
Together  with $\varphi\in \cI$ (in particular, $\lim_{t\rightarrow
\infty}\varphi(t)=\infty$), one has
\begin{eqnarray*}
     \sup_{Q\in \cBh}\sum_{i=1}^m \lambda_i \varphi(h_Q(u_i)) &\geq&   \sum^m_{i=1}\varphi\left(h_{\widetilde{L}_\epsilon/f(\epsilon)}(u_i) \right)\cdot\mu(\{u_i\})\\&
    \geq& \varphi\left( h_{\widetilde{L}_\epsilon/f(\epsilon)}(u_1)\right)\cdot \mu(\{u_1\})\\
   &=& \varphi\left(f(\epsilon)^{-1} \right)\cdot \mu(\{u_1\}) \rightarrow \infty,
\end{eqnarray*} as $\epsilon\rightarrow 0^+$, which follows from the fact that $\lim_{\epsilon\rightarrow 0^+} f(\epsilon)=0$.

When $\varphi\in \cD$, let  $\overline{L}_{\epsilon}=\widetilde{L}_{\epsilon}^\circ=\widetilde{L}_{1/\epsilon}$. Hence $\overline{L}_{\epsilon}^\circ\subset\ball$ for all $\epsilon\in (0, 1)$. We claim that $\hG(\overline{L}_{\epsilon}^\circ)\rightarrow 0$ as $\epsilon\rightarrow 0^+$. To this end, it can be checked that $$\rho_{\overline{L}_{\epsilon}^\circ}(u)=\frac{\epsilon}{\sqrt{[(T^tu)_n]^2+\epsilon^2(1-[(T^tu)_n]^2)}},$$ where $(T^tu)_n$ denotes the $n$-th coordinate of $T^tu$. Clearly $\rho_{\overline{L}_{\epsilon}^\circ}(u)\leq 1$ for all $u\in \sphere$ and $\rho_{\overline{L}_{\epsilon}^\circ}(u)\rightarrow 0$ as $\epsilon\rightarrow 0^+$ for all $u\in \eta$, where $\eta=\{u\in \sphere:  (T^tu)_n \neq 0\}.$ Also note that the spherical measure of $\sphere\setminus \eta$ is $0$.

On the other hand, $\overline{L}_{\epsilon}^\circ$ is increasing (in the sense
of set inclusion) and hence $\hG(\overline{L}_{\epsilon}^\circ)$ is strictly
increasing on $\epsilon$ due to Proposition
\ref{properties-G-hat}. To show that
$\hG(\overline{L}_{\epsilon}^\circ)\rightarrow 0$ as $\epsilon\rightarrow 0^+$,
we assume that $\inf_{\epsilon>0}\hG(\overline{L}_{\epsilon}^\circ)=\beta>0$
and seek for contradictions. By Proposition
\ref{equivalent-form-1-23}, one has, for all $\epsilon\in (0,1)$,
\begin{equation}\label{bounded-ee-1-28}
\int_{\sphere}G\bigg(\frac{\rho_{\overline{L}_\epsilon^\circ}(u)}{\beta},u\bigg)\,du\geq
\int_{\sphere}G\bigg(\frac{\rho_{\overline{L}_\epsilon^\circ}(u)}{\hG({\overline{L}_\epsilon^\circ})},u\bigg)\,du=1.\end{equation}
Moreover, as $\rho_{\overline{L}_{\epsilon}^\circ}(u)\leq 1$ for all $u\in
\sphere$, one has, for all $u\in \sphere$,
$$G\bigg(\frac{\rho_{\overline{L}_\epsilon^\circ}(u)}{\beta},u\bigg)\leq
G\bigg(\frac{1}{\beta},u\bigg).$$ Together with
(\ref{bounded-ee-1-28}) and the  dominated convergence theorem, one
gets that    \begin{equation*} 1\leq \ \lim_{\epsilon\rightarrow 0^+}
\int_{\sphere}G\bigg(\frac{\rho_{\overline{L}_\epsilon^\circ}(u)}{\beta},u\bigg)\,du
\ = \int_{\sphere}  \lim_{\epsilon\rightarrow 0^+}
G\bigg(\frac{\rho_{\overline{L}_\epsilon^\circ}(u)}{\beta},u\bigg)\,du=0. \end{equation*} This implies
$\hG(\overline{L}_{\epsilon}^\circ)\rightarrow 0$ as $\epsilon\rightarrow 0^+$. Again,
$\overline{L}_\epsilon/f(\epsilon)\in \cBh$ and $h_{\overline{L}_\epsilon/f(\epsilon)}(u_1)=f(\epsilon)^{-1},$ where  $$f(\epsilon)=\frac{\hG(\ball)}{\hG({\overline{L}_\epsilon^\circ})} \rightarrow \infty  \ \ \text{as}\ \
\epsilon\rightarrow 0^+.$$
Together with $\varphi\in \cD$ (in particular, $\lim_{t\rightarrow 0^+}\varphi(t)=\infty$), one has
\begin{eqnarray*}
     \sup_{Q\in \cBh}\sum_{i=1}^m \lambda_i \varphi(h_Q(u_i))
    \geq \varphi\left( h_{\overline{L}_\epsilon/f(\epsilon)}(u_1)\right)\cdot \mu(\{u_1\})
   = \varphi\left(f(\epsilon)^{-1} \right)\cdot \mu(\{u_1\}) \rightarrow \infty,
\end{eqnarray*} as $\epsilon\rightarrow 0^+$. This concludes the proof of ii).   \end{proof}

It is worth to mention that the argument in Proposition \ref{denial
of other possibilities} ii) for the case $\varphi \in\cD$ indeed
works for all $G\in \cG_I$ without assuming (\ref{conditionforG})
for some $q\geq n-1$. Moreover, the proof of Proposition \ref{denial
of other possibilities} can be slightly modified to show similar
results for the case $\cBt$ and  the details are omitted.

\subsection{The general dual-polar Orlicz-Minkowski problem}

In view of Proposition \ref{denial of other possibilities}, in this subsection, we will provide the continuity, uniqueness, and existence of solutions to Problem
\ref{general-dual-polar} for $\varphi\in \cI$ and with the infimum considered.

The following lemma is very useful in later context. Its proof can be found in, e.g., the proof of  \cite[Theorem 3.2]{LuoYeZhu} (slight modification is needed) and hence is omitted.
\begin{lemma} \label{bounded-for-convergence--1}  Let $\varphi\in \cI$.  Let $\mu_i, \mu$ for $i\in \N$ be nonzero finite Borel measures on $\sphere$ which are not concentrated
on any closed hemisphere and $\mu_i\rightarrow \mu$ weakly. Suppose that $\{Q_i\}_{i\geq 1}$ is a sequence of convex bodies such that  $Q_i\in \cKon$ for each $i\in \N$ and $$ \sup_{i\geq 1} \bigg\{\int_{\sphere}\varphi(h_{Q_i}(u))d\mu_i(u)\bigg\}<\infty. $$ Then $\{Q_i\}_{i\geq 1}$ is a bounded sequence in $\cKon$.
\end{lemma}

 The continuity of the extreme values for Problem
\ref{general-dual-polar} is given below.

\begin{theorem}\label{continuous theorem-1-1}
 Let $\mu_i, \mu$ for $i\in \N$ be finite Borel measures on $\sphere$ which are not concentrated
on any closed hemisphere and $\mu_i\rightarrow \mu$ weakly. Let
$G\in \cG_I$ be a continuous function such that (\ref{conditionforG}) holds for some $q\geq n-1$ and
$\varphi \in \mathcal{I}$. The following statements hold true.

\vskip 2mm \noindent i) If for each $i\in \N$, there exists $\widetilde{M}_i\in \cBt$ such that  \begin{equation}\label{Problem-1-28-1}
\int_{\sphere}\varphi\big(h_{\widetilde{M}_i}(u)\big)d\mu_i (u)=\inf\left\{ \int_{\sphere}\varphi(h_{Q}(u))d\mu_i(u): Q\in \cBt \right\},
 \end{equation}  then there exists  $\widetilde{M}\in \cBt$ such that  \begin{equation}\label{Problem-1-28-21}
\int_{\sphere}\varphi\big(h_{\widetilde{M}}(u)\big)d\mu (u)=\inf\left\{ \int_{\sphere}\varphi(h_{Q}(u))d\mu(u): Q\in \cBt \right\}.
 \end{equation} Moreover, \begin{equation} \label{limit-optimal-1-29} \lim_{ i\rightarrow \infty}\int_{\sphere}\varphi\big(h_{\widetilde{M}_i}(u)\big)d\mu_i (u)= \int_{\sphere}\varphi\big(h_{\widetilde{M}}(u)\big)d\mu (u).\end{equation}
 ii)  If for each $i\in \N$, there exists $\widehat{M}_i\in \cBh$ such that  \begin{equation*}
\int_{\sphere}\varphi\big(h_{\widehat{M}_i}(u)\big)d\mu_i (u)=\inf\left\{ \int_{\sphere}\varphi(h_{Q}(u))d\mu_i(u): Q\in \cBh \right\},
 \end{equation*}  then there exists  $\widehat{M}\in \cBh$ such that  \begin{equation*}
\int_{\sphere}\varphi\big(h_{\widehat{M}}(u)\big)d\mu (u)=\inf\left\{ \int_{\sphere}\varphi(h_{Q}(u))d\mu(u): Q\in \cBh \right\}.
 \end{equation*}  Moreover, $$\lim_{ i\rightarrow \infty}\int_{\sphere}\varphi\big(h_{\widehat{M}_i}(u)\big)d\mu_i (u)= \int_{\sphere}\varphi\big(h_{\widehat{M}}(u)\big)d\mu (u).$$\end{theorem}

\begin{proof} For each $i\in \N$, let $$\mu_i(\sphere)=\int_{\sphere} \,d\mu_i\ \ \ \ \text{and} \ \ \ \int_{\sphere}\,d\mu=\mu(\sphere).$$
 i) It can be easily checked from (\ref{Problem-1-28-1})  and $\ball \in \cBt$ that for each $i\in \N$,
$$\int_{\sphere}\varphi\big(h_{\widetilde{M}_i}(u)\big)d\mu_i (u)\leq \varphi(1)\mu_i(\sphere).$$
Moreover, the weak convergence of $\mu_i\rightarrow \mu$ yields
$\mu_{i}(\sphere)\rightarrow \mu(\sphere)$. Hence, $$ \sup_{i\geq 1}
\bigg\{\int_{\sphere}\varphi(h_{\widetilde{M}_i}(u))d\mu_i(u)\bigg\}<\infty.
$$ By Lemma \ref{bounded-for-convergence--1}, one sees that
$\{\widetilde{M}_i\}_{i\geq 1}$ is a bounded sequence in $\cKon$. As
$\widetilde{M}_i\in \cBt$ for each $i\in \N$, Lemma
\ref{lemmaforinterior} implies that there exist a subsequence
$\{\widetilde{M}_{i_j}\}_{j\geq 1}$ of $\{\widetilde{M}_i\}_{i\geq
1}$ and a convex body $\widetilde{M}\in \cBt$ such that
$\widetilde{M}_{i_j}\rightarrow \widetilde{M}.$

Now we verify that $\widetilde{M}$ satisfies the desired properties.
First of all, for any given $Q\in \cBt$, one has, for each $j\in
\N$,
$$\int_{\sphere}\varphi\big(h_{\widetilde{M}_{i_j}}(u)\big)d\mu_{i_j}
(u) \leq \int_{\sphere}\varphi\big(h_{Q}(u)\big)d\mu_{i_j} (u).$$
Together with the weak convergence of $\mu_i\rightarrow \mu$, Lemma \ref{uniformly converge-lemma},
$\varphi\in \cI$, and $\widetilde{M}_{i_j}\rightarrow
\widetilde{M}$, one obtains that
$\varphi\big(h_{\widetilde{M}_{i_j}}\big)\rightarrow
\varphi(h_{\widetilde{M}})$ uniformly on $\sphere$ and  for each
given $Q\in \cBt$, \begin{eqnarray*}
\int_{\sphere}\varphi\big(h_{\widetilde{M}}(u)\big)d\mu (u) &=&
\lim_{ j\rightarrow
\infty}\int_{\sphere}\varphi\big(h_{\widetilde{M}_{i_j}}(u)\big)d\mu_{i_j}
(u)\\ &\leq &   \lim_{ j\rightarrow
\infty}\int_{\sphere}\varphi\big(h_{Q}(u)\big)d\mu_{i_j} (u)\\ &=&
\int_{\sphere}\varphi\big(h_{Q}(u)\big)d\mu (u). \end{eqnarray*}
Taking the infimum over $Q\in \cBt$ and together with
$\widetilde{M}\in \cBt$, one gets that \begin{eqnarray*}
\int_{\sphere}\varphi\big(h_{\widetilde{M}}(u)\big)d\mu (u) \leq
\inf_{Q\in \cBt} \bigg\{ \int_{\sphere}\varphi\big(h_{Q}(u)\big)d\mu
(u)\bigg\} \leq
\int_{\sphere}\varphi\big(h_{\widetilde{M}}(u)\big)d\mu (u).
\end{eqnarray*} Hence, $\widetilde{M}\in \cBt$ verifies
(\ref{Problem-1-28-21}).

Now let us verify (\ref{limit-optimal-1-29}). To this end, let $\{\mu_{i_k}\}_{k\geq 1}$ be an arbitrary subsequence of $\{\mu_i\}_{i\geq 1}$. Repeating the arguments  above for $\mu_{i_k}$ and $\widetilde{M}_{i_k}$ (replacing $\mu_i$ and $\widetilde{M}_{i}$, respectively),  one gets a subsequence $\{\widetilde{M}_{i_{k_j}}\}_{j\geq 1}$ of $\{\widetilde{M}_{i_k}\}_{k\geq 1}$ such that $\widetilde{M}_{i_{k_j}}\rightarrow \widetilde{M}_0\in \cBt$  and $\widetilde{M}_0$ satisfies (\ref{Problem-1-28-21}). Thus, \begin{eqnarray*}
\lim_{j\rightarrow\infty} \int_{\sphere}\varphi\Big(h_{\widetilde{M}_{i_{k_j}}}(u)\Big)d\mu_{i_{k_j}} (u)&=&
\int_{\sphere}\varphi\big(h_{\widetilde{M}_0}(u)\big)d\mu (u)\\&=&\inf\left\{ \int_{\sphere}\varphi(h_{Q}(u))d\mu(u): Q\in \cBt \right\}\\&=&\int_{\sphere}\varphi\big(h_{\widetilde{M}}(u)\big)d\mu (u),\end{eqnarray*} where the first equality  follows from Lemma \ref{uniformly converge-lemma} and the last two equalities follow from (\ref{Problem-1-28-21}). This concludes the proof of (\ref{limit-optimal-1-29}), i.e.,
$$ \lim_{ i\rightarrow \infty}\int_{\sphere}\varphi\big(h_{\widetilde{M}_i}(u)\big)d\mu_i (u)= \int_{\sphere}\varphi\big(h_{\widetilde{M}}(u)\big)d\mu (u).$$

\noindent ii) The proof of this case is almost identical to the one in i), and will be omitted.
\end{proof}

The following theorem provides the existence and uniqueness of solutions to  Problem
\ref{general-dual-polar}  for $\varphi\in \cI$ and with the infimum considered.

\begin{theorem}\label{polardualorlicz-1-29-1}
Let $\varphi\in\mathcal{I}$ and
$\mu$  be  a  nonzero finite Borel measure defined on $\sphere$ which
is not concentrated on any closed hemisphere. Let
$G\in \cG_I$ be a continuous function such that (\ref{conditionforG}) holds for some $q\geq n-1$.  Then the following statements hold.

\vskip 2mm \noindent  i) There exists a convex body $\widetilde{M}\in \cBt$ such that  \begin{equation}\label{Problem-1-29-21}
\int_{\sphere}\varphi\big(h_{\widetilde{M}}(u)\big)d\mu (u)=\inf\left\{ \int_{\sphere}\varphi(h_{Q}(u))d\mu(u): Q\in \cBt \right\}.
 \end{equation} If,  in addition, both $\varphi(t)$ and $G(t,\cdot)$ are convex on $t\in (0, \infty)$, then the solution is unique.

 \vskip 2mm  \noindent  ii) There exists a convex body $\widehat{M}\in \cBh$ such that  \begin{equation*}
\int_{\sphere}\varphi\big(h_{\widehat{M}}(u)\big)d\mu (u)=\inf\left\{ \int_{\sphere}\varphi(h_{Q}(u))d\mu(u): Q\in \cBh\right\}.
 \end{equation*}  If,  in addition, both $\varphi(t)$ and $G(t,\cdot)$ are convex on $t\in (0, \infty)$, then the solution is unique.\end{theorem}

\begin{proof} Let $\mu$ be  a  nonzero finite Borel measure defined on $\sphere$ which
is not concentrated on any closed hemisphere. Let $\mu_i$  for all $i\in \N$ be nonzero finite discrete Borel measures
defined on $\sphere$, which are not
concentrated on any closed hemisphere,  such that,  $\mu_i\rightarrow
\mu$ weakly  (see e.g., \cite{Sch}).

\vskip 2mm \noindent i) By Theorem \ref{polardualorlicz1}, for each
$i\in \N$, there exists a polytope $\widetilde{P}_i\in\cBt$ solving
(\ref{Problem-1-29-21}) with $\mu$ replaced by $\mu_i$. It  follows
from Theorem \ref{continuous theorem-1-1} that there exists a
$\widetilde{M}\in \cBt$ such that (\ref{Problem-1-29-21}) holds.

Now let us prove the uniqueness. Assume that $\widetilde{M}\in \cBt$ and $\widetilde{M}_0\in\cBt$,  such that $$\int_{\sphere}\varphi\big(h_{\widetilde{M}}(u)\big)d\mu (u)=\int_{\sphere}\varphi\big(h_{\widetilde{M}_0}(u)\big)d\mu (u)=\inf\left\{ \int_{\sphere}\varphi(h_{Q}(u))d\mu(u): Q\in \cBt \right\}.
$$
Note that both $\widetilde{M}\in \cKon$ and
$\widetilde{M}_0\in\cKon$.  Let $K_0=\frac{\widetilde{M}+\widetilde{M}_0}{2}\in \cKon$.  Then,
$$h_{K_0}=\frac{h_{\widetilde{M}}+h_{\widetilde{M}_0}}{2}  \Longrightarrow\   \rho_{K_0^\circ} =2\cdot \frac { \rho_{\widetilde{M}^\circ}\cdot \rho_{\widetilde{M}_0^\circ}}{
\rho_{\widetilde{M}^\circ}+ \rho_{\widetilde{M}_0^\circ}},$$
following from $h_K\cdot
\rho_{K^\circ}=1$ for all $K\in \cKon$. The facts
that $G(t, \cdot)$ is convex and $G\in \cG_I$  is strictly increasing, together with
$\widetilde{M}\in \cBt$ and $\widetilde{M}_0\in\cBt$,   yield that
\begin{eqnarray}
\dveV(K_0^\circ) &=& \int_{\sphere}G\big(\rho_{K_0^\circ}(u), u\big)\,du \nonumber \\ &\leq& \int_{\sphere}
G\bigg(2\cdot \frac { \rho_{\widetilde{M}^\circ}(u)\cdot \rho_{\widetilde{M}_0^\circ}(u)}{
\rho_{\widetilde{M}^\circ}(u)+ \rho_{\widetilde{M}_0^\circ}(u)},\ u\bigg)\,du  \nonumber \\ &\leq& \int_{\sphere}
G\bigg(\frac{ \rho_{\widetilde{M}^\circ}(u)+
\rho_{\widetilde{M}_0^\circ}(u)}{2}, u\bigg)\,du  \nonumber  \\ &\leq&
\int_{\sphere}\frac{G\big(\rho_{\widetilde{M}^\circ}(u), u\big)+G\big(\rho_{\widetilde{M}_0^\circ}(u), u\big)}{2}\,du  \nonumber  \\ &=& \frac{\dveV(\widetilde{M}^\circ)+\dveV(\widetilde{M}^\circ_0)}{2} =\dveV(\ball). \label{equality-03-03-19}
\end{eqnarray}  Again, as $G\in \cG_I$, one
can find a constant $t_2\geq 1$ such that $\dveV(t_2K_0^\circ)
=\dveV(\ball)$ and $K_0/t_2\in
\cBt$. Due to $t_2\geq 1$ and the facts that $\varphi\in\mathcal{I}$
is convex and strictly increasing, one has
 \begin{eqnarray}
 \int_{\sphere} \varphi(h_{K_0/t_2}(u))\,d\mu(u)
  &\geq &  \inf\left\{ \int_{\sphere}\varphi(h_{Q}(u))d\mu(u): Q\in \cBt \right\} \nonumber \\ &=&
\frac{1}{2}\bigg( \int_{\sphere}\varphi\big(h_{\widetilde{M}}(u)\big)d\mu (u)+\int_{\sphere}\varphi\big(h_{\widetilde{M}_0}(u)\big)d\mu (u)\bigg) \nonumber \\
  &\geq& \int_{\sphere}\varphi\bigg(\frac{h_{\widetilde{M}}(u)+h_{\widetilde{M}_0}(u)}{2}\bigg)d\mu(u) \nonumber \\
  &=&  \int_{\sphere}\varphi(h_{K_0}(u))\,d\mu(u) \nonumber \\ &\geq& \int_{\sphere}\varphi(h_{K_0/t_2}(u))\,d\mu(u). \label{equality-03-03-19-1}
 \end{eqnarray}
Hence  all ``$\geq$" in (\ref{equality-03-03-19-1}) become ``$=$"; and this can happen if and only
if $t_2=1$ as $\varphi$ is strictly increasing. This in turn yields that all ``$\geq$"  in (\ref{equality-03-03-19}) become ``$=$" as well. In particular, as $G(t, \cdot)$ is strictly increasing,   for all $u\in \sphere$,
$$2\cdot \frac { \rho_{\widetilde{M}^\circ}(u) \cdot \rho_{\widetilde{M}_0^\circ}(u)}{
\rho_{\widetilde{M}^\circ}(u)+ \rho_{\widetilde{M}_0^\circ}(u)}=\frac{\rho_{\widetilde{M}^\circ}(u)+
\rho_{\widetilde{M}_0^\circ}(u)}{2}$$ and hence $\rho_{\widetilde{M}^\circ}(u)= \rho_{\widetilde{M}_0^\circ}(u)$ for all $u\in \sphere$. That is,
$\widetilde{M}=\widetilde{M}_0$ and the  uniqueness follows.

\vskip 2mm \noindent ii) The proof of this case is almost identical to the one in i), and will be omitted.
\end{proof}

  The following result states that the continuity of solutions to  Problem
\ref{general-dual-polar} for $\varphi\in \cI$ and with the infimum considered.

\begin{corollary}\label{con-4-12-19}
 Let $\mu_i, \mu$ for $i\in \N$ be nonzero finite Borel measures on $\sphere$ which are not concentrated
on any closed hemisphere and $\mu_i\rightarrow \mu$ weakly. Let
$G\in \cG_I$ be a continuous function such that $G(t, \cdot)$ is convex on $t\in (0, \infty)$ and (\ref{conditionforG}) holds for some $q\geq n-1$.  Let
$\varphi \in \mathcal{I}$ be convex. The following statements hold true.

\vskip 2mm \noindent i) Let $\widetilde{M}_i\in \cBt$ for each $i\in\N$ and $\widetilde{M}\in \cBt$ be the solutions to the optimization problem (\ref{Problem-1-1-1-1}) with the infimum considered for measures $\mu_i$ and $\mu$, respectively. Then $\widetilde{M}_i\rightarrow \widetilde{M}$ as $i\rightarrow \infty$.

\vskip 2mm \noindent ii) Let $\widehat{M}_i\in \cBh$ for each $i\in\N$ and $\widehat{M}\in \cBh$ be the solutions to the optimization problem (\ref{Problem-1-1-2-2}) with the infimum considered for measures $\mu_i$ and $\mu$, respectively. Then $\widehat{M}_i\rightarrow \widehat{M}$ as $i\rightarrow \infty$.
\end{corollary}
\begin{proof}
i) The proof of this result follows from the combination of the
proof of Theorem \ref{continuous theorem-1-1} and the uniqueness in
Theorem \ref{polardualorlicz-1-29-1}. Indeed, let
$\{\widetilde{M}_{i_k}\}_{k\geq 1}$ be an arbitrary subsequence of
$\{\widetilde{M}_i\}_{i\geq 1}$. Like in  the proof of Theorem
\ref{continuous theorem-1-1}, one can check that   there exist a
subsequence $\{\widetilde{M}_{i_{k_j}}\}_{j\geq 1}$ of
$\{\widetilde{M}_{i_k}\}_{k\geq 1}$ and a convex body
$\widetilde{M}_0\in \cBt$ such that
$\widetilde{M}_{i_{k_j}}\rightarrow \widetilde{M}_0.$ Moreover,
$\widetilde{M}_0$ satisfies that
$$\int_{\sphere}\varphi\big(h_{\widetilde{M}_0}(u)\big)d\mu
(u)=\inf\left\{ \int_{\sphere}\varphi(h_{Q}(u))d\mu(u): Q\in \cBt
\right\}.$$ The uniqueness in Theorem \ref{polardualorlicz-1-29-1}
yields $\widetilde{M}_0=\widetilde{M}.$

In other words, one shows that every subsequence $\{M_{i_k}\}_{k\geq 1}$ of $\{M_i\}_{i\geq 1}$ must have a subsequence $\widetilde{M}_{i_{k_j}}$ convergent to $\widetilde{M}.$ This concludes that $\widetilde{M}_{i}\rightarrow \widetilde{M}.$

\vskip 2mm \noindent ii) The proof of this case is almost identical to the one in i), and will be omitted.
\end{proof}

 \section{Variations of the general dual-polar Orlicz-Minkowski problem}\label{Section:5}

 Problem \ref{general-dual-polar} discussed in Section \ref{section-4} are only typical examples of the polar Orlicz-Minkowski type problems. In this section, several variations of Problem
\ref{general-dual-polar} will be provided.

\subsection{The general dual-polar Orlicz-Minkowski problem associated with the  Orlicz norms} \label{section 5-1}
Let $\mu$ be a given nonzero finite Borel measure defined on $\sphere$.  For  $\varphi\in \cI\cup \cD$ and for $Q\in \cKon$,  the functional $\int_{\sphere}\varphi(h_Q)\,d\mu$ is in general not homogeneous.  However, like the definition for $\hG(\cdot)$, one can define a homogeneous functional for $Q\in \cKon$ as follows: \begin{eqnarray*}
\|h_Q\|_{\mu,\varphi}&=&\inf\bigg\{\lambda>0:~\frac{1}{\mu{(\sphere)}}
\int_{\sphere}\varphi\left(\frac{h_Q(u)}{\lambda}\bigg)\,d\mu(u)\leq
1\right\} \ \ \ \text{if} \ \ \varphi\in\mathcal{I}; \\
\|h_Q\|_{\mu,\varphi}&=&\inf\bigg\{\lambda>0:~\frac{1}{\mu{(\sphere)}}
\int_{\sphere}\varphi\left(\frac{h_Q(u)}{\lambda}\bigg)\,d\mu(u)\geq
1\right\} \ \ \ \text{if} \ \ \varphi\in\mathcal{D}.
 \end{eqnarray*} For convenience, $\|h_Q\|_{\mu,\varphi}$ is called  the ``Orlicz norm" of $h_Q$, although in general it may not satisfy the triangle inequality.   Following the proof of Proposition \ref{equivalent-form-1-23}, it can be checked that, for any $Q\in \cKon$ and $\varphi\in \cI\cup\cD$, $\|h_Q\|_{\mu,\varphi}>0$ satisfies  \begin{equation}\label{propertyfornorm}
\frac{1}{\mu{(\sphere)}}\int_{\sphere}\varphi\left(\frac{h_Q(u)}{\|h_Q\|_{\mu,\varphi}}\right)\,d\mu=1.
\end{equation} Moreover,  $\|1\|_{\mu,\varphi}=1$, $\|ch_Q\|_{\mu,\varphi}=c\|h_Q\|_{\mu,\varphi}$
for any constant $c>0$ and for any $Q\in \cKon$, and
$\|h_Q\|_{\mu,\varphi}\leq\|h_L\|_{\mu,\varphi}$ for $Q, L\in \cKon$ such that $Q\subseteq L.$

The following lemma for $\varphi\in \cI\cup\cD$ can be proved similar to the proof of Proposition \ref{properties-G-hat} ii). For completeness, we provide a brief proof here.  See e.g.,  \cite[Lemma  4]{HLYZ-2010} and \cite[Lemma 3.4 and Corollary 3.5]{Huanghe} for similar results.

\begin{lemma}\label{convergneceofmeasure}
 Let $Q_i, Q\in \cKon$  for each $i\in \N$, and  $\mu_i, \mu$  for each $i\in \N$ be nonzero finite Borel measures on $\sphere$.  If     $Q_i\rightarrow Q$ and $\mu_i\rightarrow \mu$ weakly, then for all $\varphi\in \cI\cup\cD$,
    $$\lim_{i\rightarrow \infty} \|h_{Q_i}\|_{\mu_i,\varphi}=\|h_{Q}\|_{\mu,\varphi}.$$
\end{lemma}
\begin{proof}
We only prove the case for $\varphi\in \cI$ (and the case for
$\varphi\in \cD$ follows along the same lines). Let  $Q_i\in \cKon$
for all $i\in \N$ and $Q_i\rightarrow Q\in \cKon$.  Let the
constants $0<r_Q<R_Q<\infty$ be such that $r_Q\leq h_Q\leq R_Q$ and
$r_Q\leq h_{Q_i}\leq R_Q$ for all $ i\in \N. $  It can be checked
that $$r_Q\leq \ \inf_{i\geq 1} \|h_{Q_i}\|_{\mu_i,\varphi}\ \leq \
\sup_{i\geq 1} \|h_{Q_i}\|_{\mu_i,\varphi}\ \leq R_Q. $$ Assume that
$\limsup_{i\rightarrow \infty}
\|h_{Q_i}\|_{\mu_i,\varphi}>\|h_{Q}\|_{\mu,\varphi}$. There exists a
subsequence $\{Q_{i_j}\}$ of $\{Q_i\}$ such that $\lim_{j\rightarrow
\infty}
\|h_{Q_{i_j}}\|_{\mu_{i_j},\varphi}>\|h_{Q}\|_{\mu,\varphi}.$
Together with (\ref{propertyfornorm}), Lemma \ref{uniformly
converge-lemma}, the uniform convergence of  $h_{Q_i}\rightarrow
h_Q$ on $\sphere$, and the weak convergence of $\mu_i\rightarrow
\mu$,  one has \begin{eqnarray*}
 1&=& \lim_{j\rightarrow \infty} \frac{1}{\mu_{i_j}{(\sphere)}}\int_{\sphere}\varphi\left(\frac{h_{Q_{i_j}}(u)}{\|h_{Q_{i_j}}\|_{\mu_{i_j},\varphi}}\right)\,d\mu_{i_j} \\ &=& \frac{1}{\mu{(\sphere)}}\int_{\sphere}\varphi\left(\frac{h_Q(u)}{\lim_{j\rightarrow \infty} \|h_{Q_{i_j}}\|_{\mu_{i_j},\varphi}}\right)\,d\mu \\  &<&\frac{1}{\mu{(\sphere)}}\int_{\sphere}\varphi\left(\frac{h_Q(u)}{\|h_Q\|_{\mu,\varphi}}\right)\,d\mu=1.
\end{eqnarray*}
This is a contradiction and hence $\limsup_{i\rightarrow \infty} \|h_{Q_i}\|_{\mu_i,\varphi}\leq \|h_{Q}\|_{\mu,\varphi}$.   Similarly, $\liminf_{i\rightarrow \infty} \|h_{Q_i}\|_{\mu_i,\varphi}\geq \|h_{Q}\|_{\mu,\varphi}$  also holds, which leads to
$\lim_{i\rightarrow \infty} \|h_{Q_i}\|_{\mu_i,\varphi}=\|h_{Q}\|_{\mu,\varphi}$ as desired. \end{proof}

For the convenience of later citation,  the following lemma is
given, whose proof for polytopes and discrete
measures has appeared in e.g.,
\cite{GHXY,HongYeZhang-2017,Huanghe} and is similar to the proof of Lemma \ref{bounded-for-convergence--1}. A
brief sketch of the proof is provided for completeness and for future reference.

\begin{lemma} \label{bounded-for-convergence--norm-2-2} Let $\varphi\in \cI$.  Let $\mu_i, \mu$ for $i\in \N$ be nonzero finite Borel measures on $\sphere$ which are not concentrated
on any closed hemisphere and $\mu_i\rightarrow \mu$ weakly. Suppose that $\{Q_i\}_{i\geq 1}$ is a sequence of convex bodies such that  $Q_i\in \cKon$ for each $i\in \N$ and $ \sup_{i\geq 1} \|h_{Q_i}\|_{\mu_i,\varphi}<\infty. $ Then $\{Q_i\}_{i\geq 1}$ is a bounded sequence in $\cKon$.
\end{lemma}

 \begin{proof} Let $a_+=\max\{a, 0\}$ for all $a\in \R$.  For each $i\in \N$, let  $u_i\in \sphere$ be such that $\rho_{Q_i}(u_i)=\max_{u\in\sphere}\rho_{Q_i}(u),$ and hence
 $h_{Q_{i}}(u)\geq
\rho_{Q_i}(u_i) \langle u,u_i\rangle_+$ for any $u\in\sphere$. Assume that $\{Q_i\}_{i\geq 1}$ is not bounded in $\cKon$, i.e., $\sup_{i\geq 1} \rho_{Q_i}(u_i) =\infty$. Without loss of generality, let $u_i\rightarrow v\in \sphere$ and  $\lim_{i\rightarrow \infty} \rho_{Q_i}(u_i) =\infty$.  By formula (\ref{propertyfornorm}) and  $\varphi\in\mathcal{I}$, one has for any given $C>0$, there exists $i_C\in \N$ such that for all $i>i_C$,
\begin{eqnarray*}
1
&=&\frac{1}{\mu_i{(\sphere)}}\int_{\sphere}\varphi
\left(\frac{h_{Q_{i}}(u)}{\|h_{Q_{i}}\|_{\mu_i,\varphi}}\right)\,d\mu_i(u)\\
&\geq&\frac{1}{\mu_i{(\sphere)}}\int_{\sphere}\varphi\left(\frac{
 \rho_{Q_i}(u_i) \langle
u,u_i\rangle_+}{\sup_{i\geq 1}\|h_{Q_{i}}\|_{\mu_i,\varphi}}\right)\,d\mu_i(u)\\&\geq&\frac{1}{\mu_i{(\sphere)}}\int_{\sphere}\varphi\left(\frac{
C \cdot \langle
u,u_i\rangle_+}{\sup_{i\geq 1}\|h_{Q_{i}}\|_{\mu_i,\varphi}}\right)\,d\mu_i(u).
\end{eqnarray*} By Lemma \ref{uniformly converge-lemma}, the uniform convergence of $\langle u,u_i\rangle_+\rightarrow \langle u, v\rangle_+$ on $\sphere$ as $u_i\rightarrow v$,   the weak convergence of $\mu_i\rightarrow \mu$, and $\varphi\in \cI$, one gets \begin{eqnarray*}
1 &\geq&\lim_{i\rightarrow \infty} \frac{1}{\mu_i{(\sphere)}}\int_{\sphere}\varphi\left(\frac{
C \cdot \langle
u,u_i\rangle_+}{\sup_{i\geq 1}\|h_{Q_{i}}\|_{\mu_i,\varphi}}\right)\,d\mu_i(u)\\ &=& \frac{1}{\mu{(\sphere)}}\int_{\sphere} \varphi\left(\frac{
C \cdot \langle
u,v\rangle_+}{\sup_{i\geq 1}\|h_{Q_{i}}\|_{\mu_i,\varphi}}\right)\,d\mu(u)\\&\geq& \frac{1}{\mu{(\sphere)}}\cdot \varphi\left(\frac{
C \cdot c_0}{\sup_{i\geq 1}\|h_{Q_{i}}\|_{\mu_i,\varphi}}\right)\cdot \int_{\{u\in\sphere: \langle u, v\rangle \geq c_0\}} \,d\mu(u),
\end{eqnarray*}
where $c_0>0$ is a finite constant (which always exists due to the
monotone convergence theorem and the assumption that $\mu$ is not
concentrated on any closed hemisphere) such that
$\int_{\{u\in\sphere: \langle u, v\rangle \geq c_0\}} \,d\mu(u)>0$.
Taking $C\rightarrow \infty$, the fact that $\lim_{t\rightarrow
\infty}\varphi(t)=\infty $ then yields a contradiction as $1\geq
\infty$. This concludes that $\{Q_i\}_{i\geq 1}$ is a bounded
sequence in $\cKon$. \end{proof}

 Our first variation of Problem \ref{general-dual-polar}  is the following general dual-polar Orlicz-Minkowski problem associated with the  Orlicz norms:
\begin{problem}\label{general-dual-polar-norm}   Under what conditions on a nonzero finite Borel measure $\mu$ defined
on $\sphere$, continuous functions $\varphi: (0, \infty)\rightarrow (0, \infty)$ and
$G\in \cG_I\cup \cG_d$  can we find a convex body $K\in\cKon$ solving the following optimization problems:
 \begin{eqnarray}\label{Problem-1-1-1-1-norm}
&&\inf /\sup \left\{ \|h_Q\|_{\mu,\varphi}:\ \  Q\in \cBt \right\}; \\ &&
 \label{Problem-1-1-2-2-norm}
\inf /\sup \left\{ \|h_Q\|_{\mu,\varphi}:\ \  Q\in \cBh \right\}.
 \end{eqnarray} \end{problem}

 Due to the high similarity of properties of $\int_{\sphere}\varphi(h_Q)\,d\mu$ and $\|h_Q\|_{\mu,\varphi}$,  results and their proofs in Section \ref{section-4} can be extended and adopted to Problem \ref{general-dual-polar-norm}.
 For instance, the existence of solutions to Problem \ref{general-dual-polar-norm}, if the infimum is considered, can be obtained.

\begin{theorem}\label{polardualorlicz-02-02-norm}
Let $\varphi\in\mathcal{I}$ and
$\mu$  be  a  nonzero finite Borel measure defined on $\sphere$ which
is not concentrated on any closed hemisphere. Let
$G\in \cG_I$ be a continuous function such that (\ref{conditionforG}) holds for some $q\geq n-1$.  Then the following statements hold.

\vskip 2mm \noindent  i) There exists $\widetilde{M}\in \cBt$ such that  \begin{equation}\label{Problem-02-02-norm}
\|h_{\widetilde{M}}\|_{\mu,\varphi} =\inf\Big\{\|h_{Q}\|_{\mu,\varphi} :\  Q\in \cBt \Big\}.
 \end{equation}  If,  in addition, both $\varphi(t)$ and $G(t,\cdot)$ are convex on $t\in (0, \infty)$, then the solution is unique.

 \vskip 2mm  \noindent  ii) There exists $\widehat{M}\in \cBh$ such that  \begin{equation*}
\|h_{\widehat{M}}\|_{\mu,\varphi} =\inf\Big\{\|h_{Q}\|_{\mu,\varphi} :\  Q\in \cBh \Big\}.
 \end{equation*}  If,  in addition, both $\varphi(t)$ and $G(t,\cdot)$ are convex on $t\in (0, \infty)$, then the solution is unique.
 \end{theorem}

 \begin{proof}  Only the brief proof for i) is provided and the proof for ii) follows along the same lines.  First of all, $\ball\in\cBt$, and the optimization problem (\ref{Problem-02-02-norm}) is well-defined. In particular, there exists a sequence $\{Q_i\}_{i\geq 1}$ such that each $Q_i\in \cBt$ and
 $$
\lim_{i\rightarrow \infty} \|h_{Q_i}\|_{\mu,\varphi} =\inf\Big\{\|h_{Q}\|_{\mu,\varphi} :\  Q\in \cBt \Big\}<\infty.
 $$ This further implies that $\sup_{i\geq 1}\|h_{Q_i}\|_{\mu,\varphi}<\infty$, which in turn yields the existence of a subsequence  $\{Q_{i_j}\}_{j\geq 1}$ of
$\{Q_i\}_{i\geq 1}$ and $\widetilde{M}\in\cBt$, such that $Q_{i_j}\rightarrow  \widetilde{M}$, by Lemmas \ref{lemmaforinterior} and \ref{bounded-for-convergence--norm-2-2}. It then follows from  Lemma \ref{convergneceofmeasure} that  $\lim_{i\rightarrow\infty}\|h_{Q_i}\|_{\mu,\varphi}=\lim_{j\rightarrow\infty}\|h_{Q_{i_j}}\|_{\mu,\varphi}=\|h_{\widetilde{M}}\|_{\mu,\varphi}.$ This concludes the proof, if one notices $\widetilde{M}\in \cBt$, for the existence of solutions to the optimization problem (\ref{Problem-02-02-norm}).

For the uniqueness, assume that $\widetilde{M}\in \cBt$ and
$\widetilde{M}_0\in\cBt$,  such that
\begin{equation}\label{two-equal-2-2} \|h_{\widetilde{M}}\|_{\mu,
\varphi}=\|h_{\widetilde{M}_0}\|_{\mu,
\varphi}=\inf\Big\{\|h_{Q}\|_{\mu,\varphi} :\  Q\in \cBt \Big\}.
\end{equation}
Note that $G(t, \cdot)$ is convex and $G\in \cG_I$  is strictly increasing. Let $K_0=\frac{\widetilde{M}+\widetilde{M}_0}{2}$.  By (\ref{equality-03-03-19}), there is a constant $t_2\geq 1$ such that
$\dveV(t_2K_0^\circ) =\dveV(\ball)$ and hence $K_0/t_2\in \cBt$. It follows from (\ref{propertyfornorm}), (\ref{two-equal-2-2}), $t_2\geq 1$ and  $\varphi\in\mathcal{I}$ being convex and strictly increasing that
 \begin{eqnarray*} \mu(\sphere)&=&\int_{\sphere}\varphi\bigg(\frac{h_{K_0}(u)}{\|h_{K_0}\|_{\mu,\varphi}}\bigg)\,d\mu\\ &=&\frac{1}{2} \bigg[\int_{\sphere}\varphi\bigg(\frac{h_{\widetilde{M}}}{\|h_{\widetilde{M}}\|_{\mu,\varphi}}\bigg)\,d\mu+\int_{\sphere}\varphi\bigg(\frac{h_{\widetilde{M}_0}(u)}{\|h_{\widetilde{M}_0}\|_{\mu,\varphi}}\bigg)\,d\mu\bigg]\\&\geq&\int_{\sphere}\varphi\bigg(\frac{h_{K_0}(u)}{\|h_{\widetilde{M}_0}\|_{\mu,\varphi}}\bigg)\,d\mu,
 \end{eqnarray*} and hence $\|h_{\widetilde{M}_0}\|_{\mu,\varphi}\geq \|h_{K_0}\|_{\mu,\varphi}\geq \|h_{K_0/t_2}\|_{\mu,\varphi}\geq \|h_{\widetilde{M}_0}\|_{\mu,\varphi}.$  Thus,  all ``$\geq$"  become ``$=$"; and this can happen if and only if $t_2=1$.
 This in turn yields that all ``$\geq$"  in (\ref{equality-03-03-19}) become ``$=$" as well. In particular,  $\widetilde{M}=\widetilde{M}_0$ and the  uniqueness follows. \end{proof}

  Our second example is the continuity for Problem
\ref{general-dual-polar-norm} and its solutions.

\begin{theorem}\label{continuous theorem-02-02}
 Let $\mu_i, \mu$ for $i\in \N$ be finite Borel measures on $\sphere$ which are not concentrated
on any closed hemisphere and $\mu_i\rightarrow \mu$ weakly. Let
$G\in \cG_I$ be a continuous function such that (\ref{conditionforG}) holds for some $q\geq n-1$ and
$\varphi \in \mathcal{I}$. The following statements hold true.

\vskip 2mm \noindent i) Let $\widetilde{M}_i,  \widetilde{M}\in  \cBt$,  for all $i\in \N$,  be solutions to the optimization problem (\ref{Problem-1-1-1-1-norm}), with the infimum considered, for measures $\mu_i$ and $\mu$, respectively. Then, $\lim_{ i\rightarrow \infty}\|h_{\widetilde{M}_i}\|_{\mu_i, \varphi}= \|h_{\widetilde{M}}\|_{\mu, \varphi}.$  If,  in addition, both $\varphi(t)$ and $G(t,\cdot)$ are convex on $t\in (0, \infty)$, then $\widetilde{M}_i \rightarrow \widetilde{M}$ as $i\rightarrow \infty$.

\vskip 2mm \noindent ii) Let $\widehat{M}_i,  \widehat{M}\in  \cBh$,  for all $i\in \N$,  be solutions to the optimization problem (\ref{Problem-1-1-2-2-norm}), with the infimum considered, for measures $\mu_i$ and $\mu$, respectively. Then, $\lim_{ i\rightarrow \infty}\|h_{\widehat{M}_i}\|_{\mu_i, \varphi}= \|h_{\widehat{M}}\|_{\mu, \varphi}.$  If,  in addition, both $\varphi(t)$ and $G(t,\cdot)$ are convex on $t\in (0, \infty)$, then $\widehat{M}_i \rightarrow \widehat{M}$ as $i\rightarrow \infty$. \end{theorem}

\begin{proof} Only the brief proof for i) is provided and the proof for ii) follows along the same lines. It follows from $\ball\in\cBt$, (\ref{propertyfornorm}), and $\varphi\in \cI$,  in particular $\varphi(1)=1$ that $$ \sup_{i\geq 1}
\|h_{\widetilde{M}_i}\|_{\mu_i, \varphi} \leq \sup_{i\geq 1} \|h_{\ball}\|_{\mu_i, \varphi}=1.$$ Lemma
\ref{bounded-for-convergence--norm-2-2} yields that
$\{\widetilde{M}_i\}_{i\geq 1}$ is a bounded sequence.

Let $\{\widetilde{M}_{i_k}\}_{k\geq 1}$ be an arbitrary subsequence of $\{\widetilde{M}_{i}\}_{i\geq 1}$. Lemma
\ref{lemmaforinterior} yields the existence of a subsequence
$\{\widetilde{M}_{i_{k_j}}\}_{j\geq 1}$ of $\{\widetilde{M}_{i_k}\}_{k\geq
1}$ and $ \widetilde{M}_0\in \cBt$ such that
$\widetilde{M}_{i_{k_j}} \rightarrow \widetilde{M}_0$.   Together with  the minimality of $
\|h_{\widetilde{M}_{i_{k_j}}}\|_{\mu_{i_{k_j}},\varphi}$,  Lemma
\ref{convergneceofmeasure} and the weak convergence of
$\mu_i\rightarrow \mu$ imply that
$$\|h_{\widetilde{M}_0}\|_{\mu,\varphi}=\lim_{j\rightarrow \infty}
\|h_{\widetilde{M}_{i_{k_j}}}\|_{\mu_{i_{k_j}},\varphi}\leq
\lim_{j\rightarrow\infty}
\|h_{Q}\|_{\mu_{i_{k_j}},\varphi}=\|h_{Q}\|_{\mu,\varphi},$$ for all $Q\in \cBt$.  Taking the
infimum over $Q\in \cBt$ and together with $\widetilde{M}_0\in
\cBt$, one gets that \begin{eqnarray}\label{m0-3-5-19--1}
\|h_{\widetilde{M}_0}\|_{\mu,\varphi} \leq \inf_{Q\in
\cBt}\|h_{Q}\|_{\mu,\varphi} =\|h_{\widetilde{M}}\|_{\mu,
\varphi}\leq
\|h_{\widetilde{M}_0}\|_{\mu,\varphi}. \end{eqnarray}  In conclusion, every subsequence  $\{\widetilde{M}_{i_k}\}_{k\geq 1}$ of
$\{\widetilde{M}_{i}\}_{i\geq 1}$ has a subsequence $\{\widetilde{M}_{i_{k_j}}\}_{j\geq 1}$ such that \begin{eqnarray*}
\|h_{\widetilde{M}}\|_{\mu,\varphi}=\lim_{j\rightarrow \infty}
\|h_{\widetilde{M}_{i_{k_j}}}\|_{\mu_{i_{k_j}},\varphi},  \end{eqnarray*} which implies  $\lim_{ i\rightarrow
\infty}\|h_{\widetilde{M}_i}\|_{\mu_i, \varphi}=\|h_{\widetilde{M}}\|_{\mu,
\varphi}.$

Formula (\ref{m0-3-5-19--1}) asserts that $\widetilde{M}_0\in
\cBt$ solves the optimization problem (\ref{Problem-1-1-1-1-norm})
with the infimum considered.  If,  in addition, both $\varphi(t)$
and $G(t,\cdot)$ are convex on $t\in (0, \infty)$, the uniqueness in
Theorem \ref{polardualorlicz-02-02-norm} implies  $\widetilde{M}_0=
\widetilde{M}$. In conclusion, every subsequence
$\{\widetilde{M}_{i_k}\}_{k\geq 1}$ of $\{\widetilde{M}_{i}\}_{i\geq
1}$ has a subsequence $\{\widetilde{M}_{i_{k_j}}\}_{j\geq 1}$ such
that $\widetilde{M}_{i_{k_j}}\rightarrow \widetilde{M}$.  Hence
$\widetilde{M}_i \rightarrow \widetilde{M}$ as $i\rightarrow
\infty$. \end{proof}

 An argument almost identical to Lemma \ref{polytope-lemma} shows that, if $\varphi\in \cI$ and $G\in \cG_I$ satisfying (\ref{conditionforG}) for some $q\geq n-1$, the solutions to  Problem
\ref{general-dual-polar-norm} with the infimum considered for $\mu$  being a discrete measure
defined in (\ref{discrete--1--1}) (whose support $\{u_1,\cdots, u_m\}$ is
not concentrated on any closed hemisphere) must  be polytopes with
$\{u_1,\cdots, u_m\}$ being the corresponding unit normal vectors of
their faces. Counterexamples in Proposition \ref{denial of other
possibilities} can be used to prove that the solutions to Problem
\ref{general-dual-polar-norm}
 may not exist if $\varphi\in \cI
\cup\cD$ and the supremum is considered or if $\varphi\in \cD$ and
the infimum is considered. We leave the details for readers.

\subsection{The polar Orlicz-Minkowski problem associated with the general volume}\label{section 5-2}
 Let $G: (0,\infty)\times\sphere\rightarrow(0,\infty)$ be a
continuous function. In \cite{GHWXY}, the general volume of a convex body
$K\in\cKon$, denoted by $V_{G}(K)$, is proposed to be
\begin{equation*}\label{general-dual-02-02}
\vG(K)=\int_{\sphere}G(h_{K}(u),u)\,dS_K(u),
\end{equation*} where $S_K$ denotes the surface area measure of $K$ defined on $\sphere$. Note that $\vG(K)=V(K)$ if
$G(t,u)=t/n$ for any $(t, u)\in  (0,\infty)\times\sphere$.

For each $K\in \cKon$, denote by $S(K)$ the surface area of $K$. A
fundamental inequality for $S(K)$ is the celebrated classical
isoperimetric inequality (see e.g., \cite{Sch}): \begin{equation}
\label{isoperimetric-inequality-2-5}  S(K)\geq n
\big[V(\ball)\big]^{1/n} V(K)^{\frac{n-1}{n}}. \end{equation} Define the homogeneous general volume of $K\in \cKon$, denoted
by $\vGb(K)$, as follows: for $G\in \cG_I\cup\cG_d$,
\begin{equation}\label{homogenerousnormforG-02-02}
 \frac{1}{S(K)} \int_{\sphere}G\bigg(\frac{S(K)\cdot  h_{K}(u)}{\vGb(K)},u\bigg)\,dS_K(u)=1.
\end{equation}  In particular, $\vGb(K)=V(K)$ if $G(t, u)=t/n$.  Note that $\vGb(K)$ has equivalent formulas similar to (\ref{homogenerousvolumeincreasing}) and  (\ref{homogenerousvolumedecreasing}).

Problems
\ref{general-dual-polar} and \ref{general-dual-polar-norm} can be asked for $\vG(\cdot)$ and $\vGb(\cdot)$, respectively.
\begin{problem}\label{general-polar-norm-02-02-02}   Under what conditions on a nonzero finite Borel measure $\mu$ defined
on $\sphere$, continuous functions $\varphi: (0, \infty)\rightarrow (0, \infty)$ and
$G: (0, \infty)\times \sphere \rightarrow (0, \infty)$  can we find a convex body $K\in\cKon$ solving the following optimization problems:
 \begin{eqnarray*} 
& & \inf /\sup \left\{ \|h_Q\|_{\mu,\varphi}:\  Q\in \cB \right\}  \ \  \mathrm{or}\ \ \inf /\sup \left\{ \int_{\sphere}\varphi(h_{Q}(u))\,d\mu(u):\  Q\in \cB \right\}; \\
&&  \inf /\sup \left\{ \|h_Q\|_{\mu,\varphi}:\  Q\in \cBb \right\} \ \mathrm{or}\ \ \inf /\sup \left\{\int_{\sphere}\varphi(h_{Q}(u))\,d\mu(u):\  Q\in \cBb \right\},  
 \end{eqnarray*} where $\cB$ and $\cBb$ are given by
  \begin{eqnarray*}
 \cB&=&\big\{Q\in\cKon:\  \vG{(Q^\circ)}=\vG{(\ball)}\big\}; \label{def-B-polar-2-1}\\
\cBb&=&\big\{Q\in\cK_{(o)}^n:\
\vGb(Q^\circ)=\vGb(\ball)\big\},  \ \ \ \mathrm{if} \ \ G\in \cG_I\cup \cG_d.
 \end{eqnarray*}
\end{problem}
Again, when $G=t/n$, Problem \ref{general-polar-norm-02-02-02}
becomes the polar Orlicz-Minkowski problem \cite{LuoYeZhu}. From
Sections  \ref{section-4} and \ref{section 5-1}, one sees that the
existence and continuity of solutions to Problems
\ref{general-dual-polar} and \ref{general-dual-polar-norm} are
similar, and their proofs heavily depend on Lemmas
\ref{lemmaforinterior}, \ref{bounded-for-convergence--1},
\ref{convergneceofmeasure} and
\ref{bounded-for-convergence--norm-2-2}. In particular, if  alternative
arguments of Lemma \ref{lemmaforinterior} for $\vG(\cdot)$ and
$\vGb(\cdot)$ can be established,  the desired existence and continuity
of solutions, if applicable, to Problem
\ref{general-polar-norm-02-02-02} will follow.

 Some properties for $\vG(\cdot)$ and
$\vGb(\cdot)$ are summarized in the following two propositions.

 \bp \label{properties-G-2-8-polar}   Let $G: (0, \infty)\times\sphere\rightarrow (0, \infty)$ be a continuous function. The  general volume  $\vG(\cdot)$ has the following properties.

\vskip 1mm \noindent i)   $\vG(\cdot)$ is  continuous on $\cKon$ in terms of the Hausdorff metric, that is, for any sequence $\{K_i\}_{i\geq 1}$ such that $K_i\in \cKon$ for all $i\in \N$ and $K_i\rightarrow K\in \cKon$, then  $\vG(K_i)\rightarrow \vG(K). $

\vskip 1mm\noindent ii) Let $K\in \cKon$. If $\overline{G}(t, \cdot)=t^{n-1}G(t, \cdot)\in \cG_I$, then  $\vG(tK)$ is strictly increasing on $t\in (0, \infty)$ and   $$\lim_{t\rightarrow 0^+} \vG(tK) =0 \ \ \mathrm{and} \ \ \lim_{t\rightarrow \infty} \vG(tK) =\infty;$$ while if $\overline{G}\in \cG_d$, then $\vG(tK)$ is strictly decreasing on $t\in (0, \infty)$ and $$\lim_{t\rightarrow 0^+} \vG(tK) =\infty \ \ \mathrm{and} \ \ \lim_{t\rightarrow \infty} \vG(tK) =0.$$ \ep

 \begin{proof}
The fact that $K_i\rightarrow K\in \cKon$ with $K_i\in \cKon$ for
each $i\in \N$ implies that $h_{K_i}\rightarrow h_K$ uniformly on
$\sphere$ and $S(K_i)\rightarrow S(K)$. Moreover, there exist two
positive constants $r_K<R_K$ such that $$r_K\leq h_K\leq R_K \ \ \
\text{and} \ \ \ r_K\leq h_{K_i}\leq R_K \ \ \ \text{for\ all} \ \
i\in \N. $$

  \vskip 2mm \noindent i)   As $h_{K_i}\rightarrow h_K$ uniformly on $\sphere$, one has  $G(h_{K_i}(u), u)\rightarrow G(h_K(u), u)$ also uniformly on $\sphere$.  Lemma \ref{uniformly converge-lemma}  and the well known fact that $S_{K_i}\rightarrow S_K$ weakly yield that  $\vG(K_i)\rightarrow \vG(K)$ as $i\rightarrow \infty$.

  \vskip 2mm \noindent ii) Let $K\in \cKon$. For all $t>s>0$ and all $u\in\sphere$, if
$\overline{G}\in \cG_I$ (and hence $\overline{G}(t,\cdot)$ is
strictly increasing on $t>0$), then $\vG(tK)$ is strictly increasing
on $t>0$ as follows:
\begin{eqnarray*}\vG(tK) &=&\int_{\sphere}G(h_{tK}(u),
u)\,dS_{tK}(u)\\ &=&\int_{\sphere}t^{n-1}G(t\cdot h_{K}(u),
u)\,dS_{K}(u)\\ &=&\int_{\sphere}\overline{G}(t\cdot h_{K}(u),
u)h^{1-n}_K(u)\,dS_{K}(u)\\&>&\int_{\sphere}\overline{G}(s\cdot
h_{K}(u), u)h^{1-n}_K(u)\,dS_{K}(u)=\vG(sK). \end{eqnarray*}  As
$r_K\leq h_K(u)\leq R_K$ for all $u\in \sphere$, \begin{eqnarray*}
\lim_{t\rightarrow 0^+}\vG(tK) &=& \lim_{t\rightarrow
0^+}\int_{\sphere}\overline{G}(t\cdot h_{K}(u),
u)h^{1-n}_K(u)\,dS_{K}(u) \\ &\leq&  \lim_{t\rightarrow
0^+}\int_{\sphere} r_K^{1-n}\overline{G}(t\cdot R_K, u)\,dS_{K}(u)\\
&=& \int_{\sphere}\lim_{t\rightarrow 0^+}
r_K^{1-n}\overline{G}(t\cdot R_K, u)\,dS_{K}(u)=0,\end{eqnarray*}
where we have used the dominated convergence theorem and the fact
that $\lim_{t\rightarrow 0^+} \overline{G}(t, \cdot)=0$. This proves
that  $\lim_{t\rightarrow 0^+}\vG(tK)=0.$ Similarly,
$\lim_{t\rightarrow \infty}\vG(tK)=\infty$ can be proved as follows:
\begin{eqnarray*}  \lim_{t\rightarrow \infty}\vG(tK) \geq
\liminf_{t\rightarrow \infty} \int_{\sphere}\overline{G}(t\cdot
r_{K}, u)R^{1-n}_K\,dS_{K}(u)\geq \int_{\sphere}
\liminf_{t\rightarrow \infty}\overline{G}(t\cdot r_{K},
u)R^{1-n}_K\,dS_{K}(u)=\infty,\end{eqnarray*} where we have used
Fatou's lemma and the fact that $\lim_{t\rightarrow
\infty}\overline{G}(t, \cdot)=\infty$. The desired result for the
case $\overline{G}\in \cG_d$ follows along the same lines.
 \end{proof}

 \bp
\label{properties-G-hat-polar} Let $G\in \cG_I\cup \cG_d$. The  homogeneous general volume $\vGb(\cdot)$ has the following properties.

\vskip 1mm \noindent i)  $\vGb(\cdot)$ is homogeneous, that is,
$\vGb(tK)=t^n \vGb(K)$ holds for all $t> 0$ and all $K\in \cKon$.

\vskip 1mm \noindent ii)   $\vGb(\cdot)$ is continuous on $\cKon$ in terms of the Hausdorff metric, that is, for any sequence $\{K_i\}_{i\geq 1}$ such that $K_i\in \cKon$ for all $i\in \N$ and $K_i\rightarrow K\in \cKon$, then  $ \vGb(K_i)\rightarrow \vGb(K). $ \ep

 \begin{proof}    i) The desired argument follows trivially from (\ref{homogenerousnormforG-02-02}), the strict monotonicity of $G$, and  the facts that $S(tK)=t^{n-1}S(K)$ and $h_{tK}=t\cdot h_K$ for all $t>0$.

  \vskip 2mm \noindent ii) Following the notations as in Proposition \ref{properties-G-2-8-polar},
we will prove the continuity for $\vGb(\cdot)$ if $G\in \cG_I$ (and
the proof for the case $G\in\cG_d$ is omitted).   It follows from
(\ref{homogenerousnormforG-02-02}) that  \begin{eqnarray*}
\label{compare-1-23-1-polar} \int_{\sphere}G\bigg(\frac{S(K_i)\cdot
r_{K}}{\vGb(K_i)},u\bigg)\,dS_{K_i}(u) \ \leq \ S(K_i)\ \leq
\int_{\sphere}G\bigg(\frac{S(K_i)\cdot
R_{K}}{\vGb(K_i)},u\bigg)\,dS_{K_i}(u).
\end{eqnarray*}  Suppose that $\inf_{i\in\mathbb{N}} \vGb(K_i)=0$,
and without loss of generality, assume that  $\lim_{i\rightarrow
\infty} \vGb(K_i)=0$.  Then for any $\varepsilon>0$, there exists
$i_{\varepsilon}\in \N$ such that $\vGb(K_i)<\varepsilon$ for all
$i>i_{\varepsilon}$. Hence, for $i\geq
i_{\varepsilon}$, \begin{eqnarray*}
\int_{\sphere}G\bigg(\frac{S(K_i)\cdot
r_{K}}{\varepsilon},u\bigg)\,dS_{K_i}(u)\ \le \
\int_{\sphere}G\bigg(\frac{S(K_i)\cdot
r_{K}}{\vGb(K_i)},u\bigg)\,dS_{K_i}(u)\ \leq \ S(K_i).
\end{eqnarray*}
A contradiction can be obtained from  Lemma \ref{uniformly
converge-lemma}, the weak convergence of  $S_{K_i}\rightarrow S_K$,
the facts that $\lim_{t\rightarrow \infty}G(t, \cdot)=\infty$ and
$S(K_i)\rightarrow S(K)$, and Fatou's lemma as follows:
\begin{eqnarray*}
S(K)&\geq& \liminf_{\varepsilon\rightarrow 0^+}\bigg[\lim_{i\rightarrow
\infty}
\int_{\sphere}G\bigg(\frac{S(K_i)\cdot r_{K}}{\varepsilon},u\bigg)\,dS_{K_i}(u)\bigg]\\
&=&\liminf_{\varepsilon\rightarrow 0^+}
\int_{\sphere}G\bigg(\frac{S(K)\cdot r_{K}}{\varepsilon},u\bigg)\,dS_{K}(u)\\
&\geq & \int_{\sphere} \liminf_{\varepsilon\rightarrow 0^+}
G\bigg(\frac{S(K)\cdot r_{K}}{\varepsilon},u\bigg)\,dS_{K}(u)=
\infty.\end{eqnarray*}
 This is impossible and hence $\inf_{i\in\mathbb{N}}
\vGb(K_i)>0.$ Similarly, $\sup_{i\in\mathbb{N}}
\vGb(K_i)<\infty.$

Now let us prove $\lim_{i\rightarrow \infty} \vGb(K_i) =\vGb(K)$.
Assume that $\vGb(K)<\limsup_{i\rightarrow \infty} \vGb(K_i).$ There exists a subsequence $\{K_{i_j}\}$ of $\{K_i\}$ such that $\vGb(K)<\lim_{j\rightarrow \infty} \vGb(K_{i_j}) \leq \sup_{i\in\mathbb{N}}
\vGb(K_i)<\infty.$ By $G\in \cG_I$,  (\ref{homogenerousnormforG-02-02}),  Lemma \ref{uniformly converge-lemma}, $S_{K_{i_j}}\rightarrow S_K$ weakly and  $h_{K_{i_j}}\rightarrow h_K>0$ uniformly on $\sphere$, one gets \begin{eqnarray*}
 S(K)&=& \lim_{j\rightarrow \infty} \int_{\sphere}G\bigg(\frac{S(K_{i_j})\cdot h_{K_{i_j}}(u)}{\vGb(K_{i_j})},u\bigg)\,dS_{K_{i_j}}(u) \\
 &=&\int_{\sphere} G\bigg(\frac{S(K)\cdot h_{K}(u)}{ \lim_{j\rightarrow \infty}\vGb(K_{i_j})},u\bigg)\,dS_{K}(u) \\
 &<& \int_{\sphere} G\bigg(\frac{S(K)\cdot h_{K}(u)}{\vGb(K)},u\bigg)\,dS_K(u)=S(K).
\end{eqnarray*} This is a contradiction and hence $\limsup_{i\rightarrow \infty} \vGb(K_i) \leq \vGb(K).$ Similarly, $\liminf_{i\rightarrow \infty} \vGb(K_i) \geq \vGb(K)$ and then the desired equality $\lim_{i\rightarrow \infty} \vGb(K_i) =\vGb(K)$ holds.  \end{proof}

 The following lemma is a replacement of Lemma \ref{lemmaforinterior}. Note that the monotonicity of $\vG(\cdot)$ and $\vGb(\cdot)$ in terms of set inclusion,  in general, may be invalid. Therefore, our proof for Lemma \ref{lemmaforinterior-polar} is quite different from the one for Lemma \ref{lemmaforinterior}.

\begin{lemma}\label{lemmaforinterior-polar} Let
$G: (0,\infty)\times\sphere\rightarrow (0, \infty)$ be a continuous function and $G_q(t, u)= \frac{G(t,u)}{t^q}$ for $q\in \R$.

 \vskip 2mm \noindent i) Suppose that  there exists a constant   $q\in (1-n, 0)$, such that,
    \begin{equation}\label{conditionforG-polar}
    \inf\Big\{G_q(t,u): \  t\geq 1 \ \ \mathrm{and}\  \ u\in \sphere\Big\} >0.
    \end{equation} If $\{Q_i\}_{i\geq 1}$ with $Q_i\in \cB$ for all $i\in \N$ is a bounded sequence, then there exist a subsequence $\{Q_{i_j}\}_{j\geq 1}$ of $\{Q_i\}_{i\geq 1}$ and $Q_0\in \cB$ such that $Q_{i_j}\rightarrow Q_0.$

\vskip 2mm \noindent ii) Let $G\in \cG_I$ satisfy (\ref{conditionforG-polar}) for some $q\geq 1$.  If $\{Q_i\}_{i\geq 1}$ with $Q_i\in \cBb$ for all $i\in \N$ is a bounded sequence, then there exist a subsequence $\{Q_{i_j}\}_{j\geq 1}$ of $\{Q_i\}_{i\geq 1}$ and $Q_0\in \cBb$ such that $Q_{i_j}\rightarrow Q_0.$
 \end{lemma}

\begin{proof} Let $\{Q_i\}_{i\geq 1}$ with $Q_i\in \cKon$ for each $i\in \N$ be bounded. There exists a finite constant
$R>0$ such that $Q_i\subset R\ball$ for all $i\in \N$, which in turn
implies $Q_i^\circ\supset \frac{1}{R}\ball$. In particular, $h_{Q_i^\circ}\geq
1/R$ for each $i\in \N$ and $S(Q_i^\circ)\geq R^{1-n}S(\ball)$ due to the monotonicity of surface area for convex bodies.

\vskip 2mm  \noindent i) Again (\ref{conditionforG-polar}) is
equivalent to: there exist finite constants $c_0, C_0>0$  such that
for $q\in(1-n,0)$,
\begin{equation}\label{conditionforG-polar-2-5}
    \inf\Big\{G_q(t,u):  \ t\geq c_0 \ \ \mathrm{and}\  \ u\in \sphere\Big\} >C_0.
    \end{equation} Let $c_0=1/R$. Then $G(t, u)\geq C_0 t^q$  for $q\in (1-n, 0)$ and  for  all $(t, u)\in [1/R, \infty)\times \sphere$. Thus,
\begin{eqnarray*} \vG(Q_i^\circ)&=&\int_{\sphere}
G(h_{Q_i^\circ}(u), u)\,dS_{Q_i^\circ}(u)\\ &\geq& C_0\cdot
S(Q_i^\circ)  \int_{\sphere} h^q_{Q_i^\circ}(u)\
\frac{1}{S(Q_i^\circ)}\,dS_{Q_i^\circ}(u)\\&\geq& C_0\cdot
S(Q_i^\circ)  \bigg( \int_{\sphere} h_{Q_i^\circ}(u)\
\frac{1}{S(Q_i^\circ)}\,dS_{Q_i^\circ}(u)\bigg)^q\\ &=& C_0\cdot
S(Q_i^\circ) \ \bigg(\frac{nV(Q_i^\circ)}{S(Q_i^\circ)} \bigg)^q
\\&\geq&  C_0\cdot   n\big(V(\ball)\big)^{\frac{1}{n}}
\big(V(Q_i^\circ)\big)^{1-\frac{1}{n}} \
\bigg(\frac{V(Q_i^\circ)}{V(\ball)} \bigg)^{\frac{q}{n}} \\ &=&
C_0\cdot   n\big(V(\ball)\big)^{\frac{1-q}{n}}
\big(V(Q_i^\circ)\big)^{\frac{n-1+q}{n}}, \end{eqnarray*} where we
have used Jensen's inequality and  the classical isoperimetric
inequality (\ref{isoperimetric-inequality-2-5}). Recall that
$\vG(Q_i^\circ)=\vG(\ball)$ for all $i\in \N$ and $1-n<q<0$, one has
\begin{eqnarray*} \sup_{i\geq 1} \big\{ V(Q_i^\circ)\big\} \leq \bigg(\frac{\vG(\ball)}{C_0\cdot
n\big(V(\ball)\big)^{\frac{1-q}{n}} }\bigg)^{\frac{n}{n-1+q}} <\infty.
  \end{eqnarray*}

Note that  $t^n/n$ satisfies  (\ref{conditionforG}).
The proof
of  Lemma \ref{lemmaforinterior} (in particular,
(\ref{contradiction--1})) can be used to get a subsequence $\{Q_{i_j}\}_{j\geq 1}$ of $\{Q_i\}_{i\geq 1}$ and  $Q_0\in \cKon$ such that $Q_{i_j}\rightarrow Q_0$ (see also
\cite[Lemma 3.2]{Lutwak1996}). Consequently
$Q_{i_j}^\circ\rightarrow Q_0^\circ$, and the continuity of
$\vG(\cdot)$ in  Proposition \ref{properties-G-2-8-polar} further
yields that $Q_0\in \cB$ following from $Q_i\in \cB$ for all $i\in \N$.

\vskip 2mm \noindent ii)  Recall that $Q_i^\circ\supset
\frac{1}{R}\ball$ for each $i\in \N$.  As $Q_i \in \cBb$ for each $i\in \N$, one has  $$c_0= \frac{R^{-n} S(\ball)}{\vGb(\ball)}\leq
\frac{S(Q_i^\circ)\cdot  h_{Q_i^\circ}(u)}{\vGb(Q_i^\circ)}. $$  It
follows from (\ref{homogenerousnormforG-02-02}),
(\ref{conditionforG-polar-2-5}) and Jensen's inequality for $q\geq
1$ that
  \begin{eqnarray*}
1&=&\frac{1}{S(Q_i^\circ) }
\int_{\sphere}G\bigg(\frac{S(Q_i^\circ)\cdot
h_{Q_i^\circ}(u)}{\vGb(Q_i^\circ)},u\bigg)\,dS_{Q_i^\circ}(u)\\
&\geq& \frac{C_0}{S(Q_i^\circ) }
\int_{\sphere}\bigg(\frac{S(Q_i^\circ)\cdot
h_{Q_i^\circ}(u)}{\vGb(\ball)}\bigg)^q\,dS_{Q_i^\circ}(u)\\ &\geq&
C_0
\bigg(\int_{\sphere}\frac{h_{Q_i^\circ}(u)}{\vGb(\ball)}\,dS_{Q_i^\circ}(u)\bigg)^q\\
&=& C_0 \bigg(\frac{nV(Q_i^\circ)}{\vGb(\ball)}\bigg)^q.
\end{eqnarray*} This further implies that $V(Q_i^\circ) \leq  n^{-1}C_0^{-1/q} \vGb(\ball)$ for each $i\in \N$.  As in i)  (the last paragraph), one gets a subsequence $\{Q_{i_j}\}_{j\geq 1}$ of $\{Q_i\}_{i\geq 1}$ and  $Q_0\in \cKon$, such that, $Q_{i_j}^\circ\rightarrow Q_0^\circ$.  The continuity of $\vGb(\cdot)$ in  Proposition \ref{properties-G-hat-polar} further yields that $Q_0\in \cBb$ following from $Q_i\in \cBb$ for all $i\in \N$.
 \end{proof}

\vskip 2mm \noindent {\bf Remark.} It can be easily checked that if (\ref{conditionforG-polar}) holds for some $q\geq 0$, Part i) of Lemma \ref{lemmaforinterior-polar}  also holds. To this end, if (\ref{conditionforG-polar}) holds for $q\geq 0$, one can verify that $2q+n-1>0$ and  \begin{eqnarray*} \inf \Big\{G_{\frac{1-n}{2}}(t, u): \ (t, u)\in [1, \infty)\times\sphere\Big\} &=&\inf \Big\{G_q(t, u)\cdot t^{\frac{2q+n-1}{2}}: \ (t, u)\in [1, \infty)\times\sphere\Big\} \\ &\geq&  \inf \Big\{G_q(t, u): \ (t, u)\in [1, \infty)\times\sphere\Big\} >0. \end{eqnarray*} Hence, (\ref{conditionforG-polar}) holds for $\frac{1-n}{2}\in (1-n, 0)$ and then Part i) of Lemma \ref{lemmaforinterior-polar}  also follows. In particular, Part i) of Lemma \ref{lemmaforinterior-polar}  works for  $G=t/n$ and $G=1$ which correspond to the volume and the surface area, respectively. Similar to the remark of Lemma \ref{lemmaforinterior}, if $G\in \cG_d$, $G$ does not satisfy (\ref{conditionforG-polar}) for some $q\geq 1$.

 The existence of solutions and the continuity of the extreme values to Problem \ref{general-polar-norm-02-02-02} for $\vG$ are stated below.

 \begin{theorem}\label{polardualminkowskvG-2-2-8}
  Let $\varphi\in\mathcal{I}$ and let $G: (0,\infty)\times\sphere\rightarrow(0,\infty)$ satisfying (\ref{conditionforG-polar}) for some $q\in (1-n, 0)$.

\vskip 2mm \noindent i) Let $\mu$ be a nonzero finite Borel measure on $\sphere$  whose support is not
concentrated on any great hemisphere. Then  there exist $M_1, M_2\in \cB$  such that
 \begin{eqnarray}\label{Problem-7-polar-2-2-8}
 \int_{\sphere}  \varphi(h_{M_1}(u))d\mu(u)&=&\inf _{Q\in \cB}  \int_{\sphere} \varphi(h_{Q}(u))\,d\mu(u);\\
  \|h_{M_2}\|_{\mu,\varphi}  &=& \inf _{Q\in \cB}   \|h_Q\|_{\mu,\varphi}. \label{Problem-7-polar-norm-2-2-8}
 \end{eqnarray}
\noindent ii) Let $\{\mu_i\}_{i=1}^\infty$ and $\mu$ be nonzero finite Borel
measures on $\sphere$ whose supports are not concentrated
on any closed hemisphere, such that, $\mu_i\rightarrow \mu$ weakly as $i\rightarrow \infty$. Then
\begin{eqnarray*}
\lim_{i\rightarrow \infty} \bigg( \inf _{Q\in \cB}  \int_{\sphere}
\varphi(h_{Q}(u))\,d\mu_i(u)\bigg) &=&  \inf _{Q\in \cB}
\int_{\sphere} \varphi(h_{Q}(u))\,d\mu(u);\\
\lim_{i\rightarrow \infty} \bigg( \inf _{Q\in \cB}
\|h_Q\|_{\mu_i,\varphi}\bigg) &=& \inf _{Q\in \cB}
\|h_Q\|_{\mu,\varphi}.
\end{eqnarray*}
 \end{theorem}
 \begin{proof} i) Note that $\ball \in \cB$ and hence the optimization problem in (\ref{Problem-7-polar-2-2-8}) is well defined. Let $\{Q_i\}_{i\geq 1}$ be the limiting sequences such that $Q_i\in \cB$ for each $i\in \N$ and  \begin{eqnarray*}\mu(\sphere)\geq\inf _{Q\in \cB}  \int_{\sphere} \varphi(h_{Q}(u))\,d\mu(u)=
\lim_{i\rightarrow\infty}  \int_{\sphere} \varphi(h_{Q_i}(u))\,d\mu(u).
 \end{eqnarray*}  It follows from Lemma \ref{bounded-for-convergence--1} that $\{Q_i\}_{i\geq 1}$ is a bounded sequence in $\cKon$. Together with Lemma \ref{lemmaforinterior-polar}, there exist a subsequence $\{Q_{i_j}\}_{j\geq 1}$ of $\{Q_i\}_{i\geq 1}$ and  $M_1\in \cB$ such that $Q_{i_j}\rightarrow M_1.$  Lemma \ref{uniformly converge-lemma} and $\varphi\in \cI$ then yield $$\int_{\sphere} \varphi(h_{M_1}(u))\,d\mu(u)= \lim_{j\rightarrow\infty}  \int_{\sphere} \varphi(h_{Q_{i_j}}(u))\,d\mu(u)  = \inf _{Q\in \cB}  \int_{\sphere} \varphi(h_{Q}(u))\,d\mu(u).  $$

 The existence of $M_2\in \cB$ that verifies (\ref{Problem-7-polar-norm-2-2-8}) can be obtained similarly, with Lemma \ref{bounded-for-convergence--1}  and Lemma  \ref{uniformly converge-lemma}  replaced by Lemma \ref{bounded-for-convergence--norm-2-2} and Lemma \ref{convergneceofmeasure}, respectively, if one notices that  \begin{eqnarray*}1\geq \inf _{Q\in \cB} \|h_{Q}\|_{\mu,\varphi}=
\lim_{i\rightarrow\infty} \|h_{Q_i}\|_{\mu,\varphi}.
 \end{eqnarray*}
 ii) First, note that from Part i), the optimization problems (\ref{Problem-7-polar-2-2-8}) and (\ref{Problem-7-polar-norm-2-2-8}) for $\mu$ and $\mu_i$ for each $i\in\N$ have solutions. The rest of the proof is almost identical to those for Theorems  \ref{continuous theorem-1-1} and \ref{continuous theorem-02-02}, with  Lemma \ref{lemmaforinterior} replaced by Lemma \ref{lemmaforinterior-polar}. \end{proof}

  Similarly, one can prove the existence of solutions and the continuity of the extreme values to Problem \ref{general-polar-norm-02-02-02} for $\vGb(\cdot)$. The proof will be omitted due to the high similarity to those in e.g., Theorem \ref{polardualminkowskvG-2-2-8}.

 \begin{theorem}
  Let $\varphi\in\mathcal{I}$ and let $G\in \cG_I$ satisfy (\ref{conditionforG-polar}) for some constant $q\geq 1$.

\vskip 2mm \noindent i) Let $\mu$ be a nonzero finite Borel measure on $\sphere$  whose support is not
concentrated on any great hemisphere. There exist $\overline{M}_1, \overline{M}_2\in \cBb$  such that
 \begin{eqnarray*}
 \int_{\sphere}  \varphi(h_{\overline{M}_1}(u))d\mu(u)=\inf _{Q\in \cBb}  \int_{\sphere} \varphi(h_{Q}(u))\,d\mu(u)\ \ \mathrm{and} \ \
  \|h_{\overline{M}_2}\|_{\mu,\varphi}  = \inf _{Q\in \cBb}   \|h_Q\|_{\mu,\varphi}.
 \end{eqnarray*}
\noindent ii) Let $\{\mu_i\}_{i=1}^\infty$ and $\mu$ be nonzero finite Borel
measures on $\sphere$ whose supports are not concentrated
on any closed hemisphere, such that, $\mu_i\rightarrow \mu$ weakly as $i\rightarrow \infty$. Then
\begin{eqnarray*}
\lim_{i\rightarrow \infty} \bigg( \inf _{Q\in \cBb}  \int_{\sphere}
\varphi(h_{Q}(u))\,d\mu_i(u)\bigg) &=& \inf _{Q\in \cBb}
\int_{\sphere} \varphi(h_{Q}(u))\,d\mu(u));\\
\lim_{i\rightarrow \infty} \bigg( \inf _{Q\in \cBb}
\|h_Q\|_{\mu_i,\varphi}\bigg) &=& \inf _{Q\in \cBb}
\|h_Q\|_{\mu,\varphi}.
\end{eqnarray*}
 \end{theorem}

 \subsection{The general Orlicz-Petty bodies} \label{section 5-3}
The classical geominimal surface area \cite{Petty1974, Petty1985} and its $L_p$ or Orlicz extensions  (see e.g., \cite{Lutwak1996, Ye2015,Ye2015c, SHG2015, ZHY2016})  are central objects in convex geometry.  When studying the properties of various geominimal surface areas, the Petty body or its generalizations play fundamental roles. In short, the Orlicz-Petty bodies are the solutions to the following optimization problems \cite{SHG2015, ZHY2016}:  \begin{eqnarray}
&& \inf\Big\{nV_{\varphi}(K,L):  L\in\cKon
 \quad \mathrm{with}\quad V(L^\circ)=V(\ball)\Big\}  \label{1-1-1-1---2};\\
 && \inf\Big\{\widehat{V}_{\varphi}(K,L):  L\in\cKon
 \quad \mathrm{with}\quad V(L^\circ)=V(\ball)\Big\},\label{1-1-1-2}
 \end{eqnarray}
where $\varphi\in \cI$,  and  $V_{\varphi}(K, L)$ and $
\widehat{V}_{\varphi} (K, L)$ are the  Orlicz $L_{\varphi}$  mixed
volumes of $K, L\in \cKon$ defined by (see e.g., \cite{GHW2014, XJL,
ZHY2016}):
\begin{eqnarray*}\label{definition-mixed-volume-2016-07-11}
V_{\varphi}(K, L)= \frac{1}{n} \int_{S^{n-1}}
\varphi\bigg(\frac{h_L(u)}{h_K(u)}\bigg)h_K(u)dS_K(u) \ \
\mathrm{and} \ \ \widehat{V}_{\varphi} (K, L)=\bigg\|
\frac{h_L}{h_K}\bigg\|_{S_K, \varphi}. \end{eqnarray*} The surface
area measure $S_K$ may be replaced by other measures; for instance,
Luo, Ye and Zhu in \cite{LuoYeZhu} obtained the $p$-capacitary
Orlicz-Petty bodies where the surface area measure is replaced by
the $p$-capacitary measure (see e.g., \cite{Cole2015,
Jenrison1996}). As explained in \cite{LuoYeZhu}, the polar
Orlicz-Minkowski problem (i.e., Problems \ref{general-dual-polar}
and \ref{general-dual-polar-norm} with $G=t^n/n$) and the
optimization problems (\ref{1-1-1-1---2}) and (\ref{1-1-1-2}) are
quite different in their general forms; however these two problems
are also very closely related. In view of their relations, we can
ask the following problem aiming to find the general Orlicz-Petty
bodies.

\begin{problem}\label{general-polar-norm-02-02-02-petty}
Let $K\in \cKon$ be a fixed convex body. Let $\mu_K$ be a nonzero
finite Borel measure associated with $K$ defined on $\sphere$,
which is not concentrated on any closed hemisphere.  Under what
conditions on continuous functions $\varphi: (0, \infty)\rightarrow
(0, \infty)$ and $G: (0, \infty)\times \sphere \rightarrow (0,
\infty)$  can we find a convex body $M\in\cKon$ solving the
following optimization problems:
 \begin{eqnarray}\label{Problem-1-polar-norm-petty}
 \inf /\sup \left\{\bigg\| \frac{h_Q}{h_K}\bigg\|_{\mu_K, \varphi}:  Q\in \cA \right\}  \   \mathrm{or}\ \inf /\sup \left\{ \int_{\sphere} \varphi\bigg(\frac{h_Q(u)}{h_K(u)}\bigg)h_K(u)d\mu_K(u): Q\in \cA \right\},
 \end{eqnarray}
 where $\cA$ is selected from the following sets: $\cBt, \cBh, \cB$ and $\cBb$.
\end{problem}

 Note that the measure $\mu_K$  assumed in Problem \ref{general-polar-norm-02-02-02-petty} includes many interesting measures, such as, the surface area measure $S_K$, the $p$-capacitary measure \cite{Cole2015, Jenrison1996}, the Orlicz $p$-capacitary measure \cite{HongYeZhang-2017},   the $L_p$ dual curvature measures \cite{LYZActa, LYZ-Lp},  the general dual Orlicz curvature measures \cite{GHWXY, GHXY, XY2017-1, ZSY2017}, and many more.

\bd \label{orlicz-petty body} Let $K\in \cKon$ be a fixed convex
body. Let $\mu_K$ be a nonzero finite Borel measure associated with $K$ defined on
$\sphere$, which is not concentrated on any
closed hemisphere. If $M\in \cA$ solving the optimization problem
(\ref{Problem-1-polar-norm-petty}), then $M$ is called a general
Orlicz-Petty body of $K$ with respect to $\mu_K$. \ed

Recall that if $K\in \cKon$, there are two constants $0<r_K<R_K$ such that $r_K\ball\subset K\subset R_K\ball$. In view of this,
the existence, continuity and uniqueness, if applicable,  of the general Orlicz-Petty bodies  with respect to $\mu_K$ can be obtained (almost identically) as in Sections \ref{section-4}, \ref{section 5-1} and \ref{section 5-2}. Polytopal solutions and counterexamples as  in Proposition \ref{denial of other possibilities}, when $K$ is a polytope, can be also established accordingly, if applicable.

\vskip 2mm \noindent {\bf Acknowledgments.}
DY is supported by a NSERC grant. BZ is supported by NSFC (No.\ 11501185) and the Doctor
Starting Foundation of Hubei University for Nationalities (No.\ MY2014B001).

\vskip 2mm \noindent Sudan Xing, \ \ \ {\small \tt sudanxing@gmail.com}\\
{ \em Department of Mathematics and Statistics,   Memorial University of Newfoundland,
    St.\ John's, Newfoundland, Canada A1C 5S7 }

\vskip 2mm \noindent Deping Ye, \ \ \ {\small \tt deping.ye@mun.ca}\\
{ \em Department of Mathematics and Statistics,
    Memorial University of Newfoundland,
    St.\ John's, Newfoundland, Canada A1C 5S7 }

\vskip 2mm \noindent Baocheng Zhu, \ \ \ {\small \tt zhubaocheng814@163.com}\\
{ \em  Department of Mathematics,
    Hubei Minzu University,
    Enshi, Hubei, China 445000}

\end{document}